%% file: main.tex
\documentclass[a4paper,reqno]{amsart}

\title{Spherical Twists for Gorenstein Orders and $G$-Hilb}
\author{Marina Godinho}

\input{2_packages.tex}

\input{1_biblatex.tex}
\input{3_amsthm.tex}

\usepackage{upref}
\crefdefaultlabelformat{#2\textup{#1}#3}
\begin{document}
% ==== ==== ==== ==== ==== ====

\begin{abstract}
    This paper constructs derived autoequivalences of Gorenstein orders as twists around spherical functors. More precisely, given a Gorenstein order $A$ and a quotient $p \colon A \to B$, then we specify natural conditions on $B$ under which the twist around the corresponding derived restriction of scalars functor is a derived autoequivalence of $A$. In the process, we show that the associated cotwist is a shift of the Nakayama functor of $B$. These results, together with local-to-global technology, are then used construct new derived autoequivalences for skew group algebras and $G$-Hilbert schemes, and we apply this theory to explicit examples.
\end{abstract}

\maketitle

\tableofcontents

\section{Introduction}

In recent years, the action of autoequivalences on  triangulated categories has not only been a key tool to study of the structure of these categories, but it has also found a wide variety of applications in in birational and enumerative geometry, mirror symmetry, as well as moduli (See, for example, \cite{BDA, BPP, GM, J, MY, Qiu} in representation theory and \cite{add, Don, BW25, DE13, DW4, ST, Toda} in algebraic geometry).

In this paper, we will construct autoequivalences of derived categories as spherical twists and cotwists. Twists around spherical objects were introduced by Seidel and Thomas \cite{ST} in order to obtain derived autoequivalences of complex smooth projective varieties. This construction was generalised by \cite{AL} and various authors to twists around spherical functors. For a survey on the development of this technology, see, for example, \cite{add} or \cite{seg}.  

Roughly, given triangulated categories $\cA$, $\cB$ and a  functor $S \colon \cA \to \cB$, then \cite{An, AL}  construct two closely related functors $T \colon \cB \to \cB$ and $C \colon \cA \to \cA$ called the \textit{twist} and \textit{cotwist} around $S$, respectively. When both $T$ and $C$ are equivalences, then $S$ is said to be a \textit{spherical functor}. 

Properties of the cotwist afford control over properties of the twist. Therefore, given an autoequivalence $T$ of a triangulated category $\cB$, it is interesting to characterise $T$ as a twist around a functor $S \colon \cA \to \cB$ whose domain category $\cA$ is simpler (e.g. $\cA$ is the derived category of a finite dimensional algebra) than the codomain category $\cB$ (e.g. $\cB$ is the category of an infinite dimensional algebra).

The setting of this paper is that of twists around ring morphisms. Suppose $p \colon A \to B$ is a surjective morphism between algebras $A$ and $B$. For now, assume that $B$ finite dimensional and is perfect as a right and left $A$-module. Then, $p \colon A \to B$ induces the restriction-extension of scalars adjoint pairs as highlighted below. 

\begin{equation*} 
    \begin{tikzcd}[column sep=10em]
        \Db(B) \arrow[rr, "F  = - \otimes^{\Lderived{}} {}_{B}  B {}_A \cong \Rderived{\Hom_{B}}{({}_A B {}_{B}, -)}" description, ""{name=F,above}] & {} & \arrow[ll, "F^{\mathrm{RA}} = \Rderived{\Hom_A}{( {}_{B} B {}_A, -)}"{name=RA}, bend left=12] \arrow[ll, "F^{\mathrm{LA}}  = - \otimes^{\Lderived{}} {}_A B {}_{B}"{name=LA}, swap, bend right=12] \Db(A)
        \arrow[phantom, from=LA, to=F, "\scriptstyle\boldsymbol{\perp}"description]
        \arrow[phantom, from=F, to=RA, "\scriptstyle\boldsymbol{\perp}"description]
    \end{tikzcd}
\end{equation*}

The twist around $F$ was characterised in \cite{Go} as follows.

\begin{thmx}[{\cite[3.8]{Go}}]
    The twist around $F$ is
    \begin{align*}
        \cT & = \Rderived{\Hom}_A({}_A \ker p {}_A, {}_A, -) \colon \Db(A) \to \Db(A).
    \end{align*}
\end{thmx}

The goal of this paper is to specify additional natural assumptions on the surjection $p \colon A \to B$ so that the functor $F$ is spherical. More specifically, we will specify conditions on $p$ under which $\cT$ is an equivalence and, moreover, the associated cotwist is a shift of the Nakayama autoequivalence of $B$. As a consequence, we apply this technology to the setting of skew-group algebras and $G$-Hilbert schemes., thus obtaining new interesting derived autoequivalences both in algebraic and geometric settings.

\subsection{Gorenstein orders}

Throughout, let $\cR$ be a commutative noetherian local ring of Krull dimension $d$ and admitting a dualizing module $\omega_{\cR}$ as in \cref{module fin}. Then, a module finite $\cR$-algebra $A$ is a \textit{Gorenstein $\cR$-order} if there is an $A$-bimodule isomorphism $\omega_A := \Rderived{\Hom_\cR}(A, \omega_\cR) \cong A$. Further details of Gorenstein orders are recalled in \S \ref{symmetric alg}, but the key property of these algebras we are interested in is that they exhibit a form of Serre duality (see \eqref{derived serre duality} and \eqref{serre duality} for more details).

In view of geometric applications, we will work within the following setup.

\begin{setx} \label[setup]{setup intro}
    Let $p \colon A \to B$ be a surjection of $K$-algebras where $A$ is a Gorenstein $\cR$-order and $B$ is finite dimensional over $K$. Suppose further that $B$ is perfect as a right $A$-module and that it either has finite global dimension or is self-injective.
\end{setx}

Our first main result specifies conditions so that the functor $F$ is an spherical.

\begin{thmx}[{\ref{finite cotwist}, \ref{spherical condition 4}}] \label{thm intro 1}
    Under \cref{setup intro}, suppose that $\Ext^k_A(B, B) = 0$ for all $k \neq 0, d$. Then $F \colon \Db(\modu B) \to \Db(\modu A)$ is spherical. Thus,
    \begin{align*}
        \cT & = \Rderived{\Hom}_A({}_A \ker p {}_A, -) \colon \Db(A) \to \Db(A).
    \end{align*}
    is an equivalence. Moreover, the cotwist around $F$ is $\cC = [-d-1] \cdot \cN$.
\end{thmx}

Since $B$ is assumed to be finite dimensional, it has finitely many simples up to isomorphism. Thus, let $\Omega = \{ S_j \}_{j=1}^k$ denote the set containing one representative from each isomorphism class of simple modules. Further, if $B$ is self-injective, let $\sigma$ be the Nakayama permutation of $B$.

Since the twist and cotwist are closely related, we obtain the following immediate corollary of \cref{thm intro 1}. 

\begin{corx}[{\ref{permutes simples}}] 
Under \cref{setup intro}, suppose that $\Ext^k_A(B, B) = 0$ for all $k \neq 0, d$. Then, the diagram
    \[ \begin{tikzcd}[column sep = 5em]
        \Db(\modu A) \arrow[r, "\cT"] & \Db(\modu A) \\
        \Db(\modu B) \arrow[u, "F"] \arrow[r, "{\cN(-)[-d+1]}"] & \Db(\modu B) \arrow[u, "F"]
    \end{tikzcd} \]
    commutes. In particular, 
    \begin{align*}
        & \Rderived{\Hom_A}(\ker p, {}_{B} B) \cong {}_{B} D B {}_A [-d + 1].
    \end{align*}
    If, further, $B$ is self-injective then
    \begin{align*}
        & \Rderived{\Hom_A}(\ker p, S_i) \cong S_{\sigma^{-1}(i)} {}_A [-d + 1].
    \end{align*}
\end{corx}

\medskip
The Ext-vanishing condition required to construct the spherical twist may seem unnatural. However, when $\dim \cR = 2, 3$, this condition is guaranteed whenever $(\ker p)^2 \cong \ker p$ (see \ref{tor vanishing} for more details). This is true, for instance, when $B$ is an idempotent quotient of $A$. We thus obtain the following corollary in dimensions $2$ and $3$.

\begin{corx} [{\ref{3d cotwist}, \ref{3d spherical}}] \label[corollary]{thm intro 2}
   Under \cref{setup intro}, suppose that $\dim \cR = 2, 3$. If $(\ker p)^2 \cong \ker p$, then, $F \colon \Db(\modu B) \to \Db(\modu A)$ is spherical. The cotwist around $F$ is $\cC \cong [-d-1] \cdot \cN$, where $\cN(-)$ denotes the derived Nakayama autoequivalence of $B$.
\end{corx}

\subsection{Skew-group algebras}

Consider a finite abelian group $G \subset \SL(3, \C)$ acting on the $\C$-algebra $R = \C[x, y, z]$. Then, this group action allows us to define the \textit{skew-group algebra} $R\#G$, the construction of which is recalled in \S \ref{sec:mckay}. These algebras exhibit a rich representation-theoretic structure. For example, they can be presented, up to a Morita equivalence, by the McKay quiver of $G$. This correspondence is often referred to as the \textit{special McKay correspondence}, which is also briefly recalled in \S \ref{sec:mckay}. 

In section \S \ref{new derived equivalences skew} we develop local-to-global technology in order to lift \cref{thm intro 2} to the setting of $R\#G$. As a consequence, we obtain the following result. 

\begin{thmx}[{\ref{spherical RG}, \ref{tw ctw RG}}] \label{thm intro 3}
    Let $p \colon R\#G \to B$ be a surjection, where $B$ is a finite dimensional $\C$-algebra which either is self-injective or has finite global dimension. Suppose that $\ker p = (\ker p)^2$. Then, the functor
    \[ F = - \Dtensor_B B \colon \Db(\modu B) \to \Db(\modu S\#G) \]
    is spherical. Hence, the twist
    \begin{align*}
    \cT = \Rderived{\Hom_{R\#G}}(\ker p, -) \colon \Db(\modu R\#G) \to \Db(\modu R\#G)
    \end{align*}
    is an autoequivalence. 
\end{thmx}

\subsection{$G$-Hilbert schemes}

On the geometric side, the action of a finite abelian group $G \subset \SL(3, \C)$ on $R = \C[x, y, z]$ allows us to construct the invariant ring $R^G$ which defines the quotient singularity $\C^3/G := \spec R^G$. 

This quotient singularity admits a crepant resolution known as $G$-$\mathrm{Hilb}(\C^3)$ \cite{Nakamura}. Moreover, \cite{BKR} prove that there is a derived equivalence
\[ \Phi_G \colon  \Db(\modu R\#G) \tow{\sim} \Db(\coh G\mathrm{-Hilb}(\C^3)) \]
thereby establishing a bridge between skew-group algebras and $G$-Hilbert schemes. Thus, autoequivalences $\Db(\modu R\#G)$ induce autoequivalences of $\Db(\coh G\mathrm{-Hilb}(\C^3))$. As a corollary of the equivalence established by \cite{BKR} and of Theorem \ref{thm intro 3} we obtain the following result in \S \ref{sec: GHilb}. 

\begin{thmx}[{\ref{quot them}}]
    Let $p \colon R\#G \to B$ be a surjection, where $B$ is a finite dimensional $\C$-algebra which either is self-injective or has finite global dimension. Suppose that $\ker p = (\ker p)^2$. Then, the functor
    \[ F_G = \Phi_G \cdot - \Dtensor_B B \colon \Db(\modu B) \to \Db(\coh G \text{-} \Hilb(\C^3)) \]
    is spherical. Thus, the twist around $F_G$
    \begin{align*}
        & \cT_G =  \Phi_G \cdot \Rderived{\Hom_{R\#G}(\ker p, -)} \cdot \Phi_G^{-1} \colon \Db(\coh G \text{-} \Hilb(\C^3)) \to \Db(\coh G \text{-} \Hilb(\C^3)).
    \end{align*}
    is an equivalence.
\end{thmx}

\subsection{Examples}

Let $n>3$ be an odd number, $\epsilon_n$ a primitive $n$th root of unity and consider the group
\[ G = \left \langle \begin{pmatrix} \epsilon_n & 0 & 0 \\ 0 & \epsilon_n & 0 \\ 0 & 0 & \epsilon_n^{n -2} \end{pmatrix}\right \rangle \]
acting on $R = \C[x, y, z]$ as $g * f(x, y, z) = f \left((x, y, z) g  \right)$. Then, in \S \ref{sec: example}, we specify precisely when idempotent quotients of $R\#G$ have finite global dimension and, thus, induce derived autoequivalences of $R\#G$. We conclude that section by explicitly computing all idempotent quotients which induce derived autoequivalences of $R\#G$ for the case $n = 7$. This result can be thought of as a statement that the categories $\Db(\modu R\#G)$ and $\Db \left(G \text{-} \Hilb(\C^3) \right)$ have large autoequivalence groups.

\subsection{Further connections to the literature}

Let $f \colon X \to \spec R$ be a crepant contraction (see e.g. \ \cite[2.1]{DW4} for the precise technical definition of contraction) with fibres of dimension at most one. Assume further that $R$ is a local isolated hypersurface singularity. Then, there is an $R$-module $R \oplus M$ and derived equivalence $\Phi_M \colon \Db(\coh X) \to \Db(\modu A)$, where $A := \End_R(R \oplus M)$ \cite[3.2.10]{VdB04}. In \cite{DW1, DW2}, Donovan and Wemyss study the  \textit{noncommutative twist} functor
\[ \cT' \colon \Rderived{\Hom_A}(Ae_0A, -) \colon \Db(\modu A) \to \Db(\modu A). \]
The authors show that it is an equivalence by proving that it is a composition of flopping equivalences. 

 \Cref{thm intro 2} recovers this result in the following way. To the contraction $f \colon X \to \spec R$, Donovan and Wemyss \cite[2.9]{DW1} associate an invariant called the \textit{contraction algebra},  which encodes much of the information associated to the contraction. 

Let $e_0$ be idempotent in $A$ corresponding to the projection $R \oplus M \to R$. Then, the contraction algebra is the idempotent quotient $A_\con := A/Ae_0A$. As will discussed in \ref{nc twist first ex}, the natural surjection $p \colon A \to A_\con$ satisfies \cref{setup intro}, and so it follows that the functor $\cT'$ is an equivalence. 

Hence, in this setting, \cref{thm intro 2}  and the local-to-global technology in \S \ref{new derived equivalences skew} as well as \cite{Go} suggest that, from the point of view of constructing derived equivalences, there might be other interesting quotients of $A$ to study. 

\subsection*{Acknowledgements} This work was carried out as a part of the author's PhD thesis, and the author would like to thank her supervisor Michael Wemyss for his helpful guidance. Moreover, many thanks goes to Ed Segal and Yukari Ito helpful suggestions which improved this work.

\subsection*{Funding} The author was partially supported by ERC Consolidator Grant 101001227 (MMiMMa).

\subsection*{Open access} For the purpose of open access, the author has applied a Creative Commons Attribution (CC:BY) licence to any Author Accepted Manuscript version arising from this submission.

\section{Conventions and summary of known results} \label{derived functors}

\subsection{Conventions}

Throughout this paper, assume that $K$ is a field. Modules will, by convention be right modules. We will write $\Mod R$ for the category of modules over the ring $R$, $\modu R$ for the category of finitely generated modules, $\proj R$ for the category of finitely generated projective modules, and $\fl R$ for the category of finite length modules. We will write $\Kcat(R) \colonequals \Kcat(\Mod R)$ and $\Dcat(R) \colonequals \Dcat(\Mod R)$ for the homotopy and derived categories of $R$, respectively.

An $R$-$S$-bimodule $L$ will sometimes be denoted ${}_R L {}_S$ and will by convention be a left $R$-module and a right $S$-module. We will often refer to an $R$-$S$-bimodule as an $S \otimes_\Z R^\op$-module. Consider bimodules ${}_R L {}_S$, ${}_R M {}_S$, and ${}_S N {}_R$. We will sometimes abbreviate the Hom-set $\Hom_S( {}_R L {}_S, {}_R M {}_S)$ as $\Hom_S( {}_R L, {}_R M)$. Similarly, the tensor product ${}_R L {}_S \, \otimes_S \, {}_S M {}_T$ will often be written as ${}_R L \otimes_S M {}_T$.

Given an $R$-module $M$, we will often denote a projective resolution 
\[ \ldots Q^{(i)} \tow{q^{(i)}} Q^{(i-1)} \to \ldots \tow{q^{(1)}} Q^{(0)} \tow{q^{0}} M \to 0 \]
as $q \colon Q \to M$.

Let $\cA$ be a category. If $f$ and $g$ are morphisms in $\cA$, then we will denote the composition "$g$ then $f$" as $f \cdot g$. Let $Q$ be a quiver and suppose $\alpha$ and $\beta$ are composable arrows in $Q$. We will denote the composition "$\alpha$ then $\beta$" as $\alpha \beta$.

Finally, given a triangulated category $\cA$ with suspension functor $\Sigma$, we will often abbreviate a distinguished triangle
\[ a \to b \to c \to \Sigma a \]
as
\[a \to b \to c \to^+. \]

\subsection{Derived functors}

A central piece of technology used in this paper is the derived hom and tensor product functors. Hence, to set conventions, we briefly recall in this section some standard constructions related to these functors. To fix notation, let $A$ and $B$ be rings throughout.

The derived hom and tensor product functors are defined via the hom-complex and the total tensor product complex, as well as the notions of K-injective and K-flat resolutions. 
    
    \begin{definition} \label{hom-cplx}
    The Hom-complex functor 
\[ \Hom_A^*(-, -) \colon \Kcat(A \otimes_\Z B^\op)^\op \times \Kcat(A) \to \Kcat(B) \]
 sends complexes $L \in \Kcat(A \otimes_\Z B^\op)$ and $M \in \Kcat(A)$ to the complex $\Hom_A^*(L, M)$ which, at degree $k$, is given by the $B$-module
\[ \Hom_A^*(L, M)^k = \prod_{p + q = k} \Hom_A(L^{-p}, M^q). \]
The differential is defined by the rule
\[ d^k(f) = d_M \cdot f + (-1)^{k+1} f \cdot d_L \]
for all $f \in  \Hom_A^*(L, M)^k$.
    \end{definition}

\begin{definition} \label{tot tensor}
     Similarly, there is a tensor product functor
\[ \tot(- \otimes_B -) \colon \Kcat(B) \times \Kcat(A \otimes_\Z B^\op) \to \Kcat(A) \]
sending complexes $L \in \Kcat(B)$ and $M \in \Kcat(A \otimes_\Z B^\op)$ to the complex $\tot(L \otimes_B M)$ which, at degree $k$, is given by the $A$-module
\[ \tot(L \otimes_B M)^k = \bigoplus_{p+q=k} L^p \otimes_B M^q. \]
with differential at $k$ defined as
\[ d^k = \sum_{p+q =k} d_L^p \otimes 1_{M^q} + (-1)^p 1_{L^p} \otimes d_M^q \]
\end{definition}

\begin{definition} Let $\cA$ be an abelian category. Recall further that  
    \begin{enumerate}
        \item A complex $I$ over $\cA$ is \textit{K-injective} if $\Hom_{\Kcat(\cA)}(C, I) = 0$  for every acyclic complex $C$. 
        \item A complex $P$ over $\cA$ is \textit{K-projective} if  $\Hom_{\Kcat(\cA)}(P, C) = 0$ for every acyclic complex $C$.
        \item A complex $Q$ of $A^\op$-modules is \textit{K-flat} if for every acyclic complex $C$ of right $A$-modules, the total complex $\tot(C \otimes_A Q)$ is acyclic.
    \end{enumerate}
\end{definition}

Given a complex $M \in \Dcat(\cA)$, then a \textit{K-injective resolution} of $M$ is a chain map $M \to I$ which is a quasi-isomorphism and where $I$ is a K-injective complex. Similarly, given a complex $M$ of $A$-modules, then a K-flat (resp.\ K-projective) resolution of $M$ is a a chain map $P \to M$ which is a quasi-isomorphism and where $P$ is a K-flat (resp.\ K-projective) complex.

\begin{definition} \label{derived hom def}
    For complexes $M \in \Dcat(A \otimes_\Z B^\op)$ and $N \in \Dcat(\Mod A)$, let $N \to I$ be a K-injective resolution. Then
    \[ \Rderived{\Hom_A}(M, N) \colonequals \Hom^*_A(M, I) \in \Dcat(B). \]
    Further, the complex $\Rderived{\Hom_A}(M, N)$ does not depend (up to quasi-isomorphism) on the choice of K-injective resolution.
\end{definition}

\medskip
Let $P \to M$ be a K-projective resolution of $M$. Then, it is a classical fact that there is an quasi-isomorphism
\[ \Rderived{\Hom_A}(M, N) \cong \Hom^*_A(Q, N) \]
which is natural in $M$ and $N$.

\begin{definition}  \label{derived tensor def}
    For complexes $M \in \Dcat(\Mod A)$ and $N \in \Dcat(\Mod B \otimes_\Z A^\op)$, let $P \to N$ be a resolution in $\Dcat(\Mod B \otimes_\Z A^\op)$ which is K-flat as a complex of $B^\op$-modules. Then
    \[ M \Dtensor_A N \colonequals \tot(M \otimes_A P) \in \Dcat(B). \] 
\end{definition}

Finally, recall that a complex $M$ of $A$-$B$-bimodules is \textit{biperfect} if both $M \in \Kb(\proj A^\op)$ and $M \in \Kb(\proj B)$.

\subsection{Self-injective algebras} \label{app: self inj}

This section recalls some classic results about self-injective algebras. 
    
An algebra $A$ is said to be \textit{self-injective} if it is injective as a module over itself. When $A$ is finite dimensional then self-injectivity is equivalent to the property that finitely generated indecomposable projective modules coincide with finitely generated indecomposable injectives \cite[IV 3.7]{FroAlg}.
    
Consider a finite-dimensional algebra $A$ viewed as the path-algebra over $K$ of a quiver $Q$ with relations $I$. Then, self-injectivity of $A$ can be thought of as a symmetry condition on $(Q, I)$, since we want the class of indecomposable injective modules to match the class of indecomposable projective modules.

    \begin{notation}
        Throughout, given a quiver $Q$ with relations $I$ and a vertex $i$, let $e_i$ denote the trivial idempotent at vertex $i$. Write $P(i) = e_i KQ/I$ and $I(i) = \Hom_K(e_i KQ/I, K)$. Finally, let $S(i)$ denote the simple module at vertex $i$. 
        
        Then, If $Q_0$ is the set of vertices of $Q$, then the complete set of indecomposable projective, injective and simple $KQ/I$-modules are $\{ P(i) \mid i \in Q_0\}$ and $\{ I(i) \mid i \in Q_0\}$, $\{ S(i) \mid i \in Q_0\}$, respectively.
    \end{notation}

    \begin{example} \label{self inj examples} Let $Q_1$ be the quiver
        
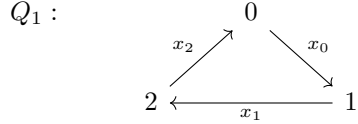
\[
        \begin{tikzcd}
            Q_1: & {}& 0 \arrow[dr, "x_0"] & {} \\
            {} & 2 \arrow[ur, "x_2"] & {} & 1 \arrow[ll, "x_1"]
        \end{tikzcd}
        \]
        Consider $Q_1$ with relations $I_1 = \langle x_0 x_1 x_2, x_1 x_2 x_0, x_2 x_0 x_1 \rangle$. Then, the indecomposable projective and injective representations are, up to isomorphism,
        
        \[ \begin{tikzcd}[sep=small]
            I(2) \cong P(0): &[-2em] {}& K \arrow[dr, "1"] & {} \\
            {} & K \arrow[ur, "0"] & {} & K \arrow[ll, "1"]
        \end{tikzcd}
        \begin{tikzcd}[sep=small]
           I(0) \cong P(1): &[-2em] {} & K \arrow[dr, "0"] & {} \\
           {} &  K \arrow[ur, "1"] & {} & K \arrow[ll, "1"]
        \end{tikzcd}
        \begin{tikzcd}[sep=small]
            I(1) \cong P(2): &[-2em] {}& K \arrow[dr, "1"] & {} \\
            {} & K \arrow[ur, "1"] & {} & K. \arrow[ll, "0"]
        \end{tikzcd}
        \]
        Thus, the path algebra of $Q_1$ subject to relations $I_1$ is self-injective. 
    \end{example}
    
    \begin{definition} \label{Nakayama perm}
        Let $A$ be a finite-dimensional self-injective algebra. Then, by \cite[IV 3.7]{FroAlg}, $P(i)$ is also an indecomposable injective $A$-module. It follows that its socle
        \[ \soc P(i) = \sum_{N \trianglelefteq P(i) \mid N \text{is simple} } N \]
        is a simple module which, therefore, must be $S(j)$ for some $j \in Q_0$. The permutation of $Q_0$ defined by $\sigma \colon i \mapsto j$ is called the \textit{Nakayama permutation} of $A$.
    \end{definition}
        
    \begin{example}
    Consider the example \ref{self inj examples}. Then, since $P(0) \cong I(2)$, the socle of $P(0)$ is $S(2)$. That is, $\sigma(0) = 2$. Similarly, $\sigma(1) = 0$ and $\sigma(2) = 1$. 
    \end{example}
    
    When $A$ is is self-injective and finite dimensional, then the category $\modu A$ of finitely generated modules over $A$ is equipped with an automorphism 
    \[ \cN \colonequals - \otimes_A D A \colon \modu A \to \modu A \]
    called the \textit{Nakayama functor}. Here, the dual $D A = \Hom_K(A, K)$ is equipped with its usual bimodule structure: for $f \in D(A)$ and $x \in A$, $(a \cdot f \cdot b) (x) = f(bxa)$. 
    
    Since for finite dimensional algebras finitely generated indecomposable projective modules coincide with finitely generated indecomposable injectives it follows that $\cN$ is exact and extends to an automorphism of $\Db(\modu A)$.    

\newpage
\subsection{Twist and cotwist around ring morphisms} \label{twist res scalars}

\subsubsection{Setting}

Throughout, unless mentioned otherwise, let $p \colon A \to B$ be a ring morphism, and assume that $B$ is perfect as an $A$-module. Then, the morphism $p \colon A \to B$ induces the restriction-extension of scalars adjoint pairs $(F^\LA, F)$, $(F, F^\RA)$ on the derived category, as highlighted below. 

\begin{equation} \label{res-ext diagram}
    \begin{tikzcd}[column sep=10em]
        \Dcat(B) \arrow[rr, "F  = - \otimes^{\Lderived{}} {}_{B}  B {}_A \cong \Rderived{\Hom_{B}}{({}_A B {}_{B}, -)}" description, ""{name=F,above}] & {} & \arrow[ll, "F^{\mathrm{RA}} = \Rderived{\Hom_A}{( {}_{B} B {}_A, -)}"{name=RA}, bend left=12] \arrow[ll, "F^{\mathrm{LA}}  = - \otimes^{\Lderived{}} {}_A B {}_{B}"{name=LA}, swap, bend right=12] \Dcat(A)
        \arrow[phantom, from=LA, to=F, "\scriptstyle\boldsymbol{\perp}"description]
        \arrow[phantom, from=F, to=RA, "\scriptstyle\boldsymbol{\perp}"description]
    \end{tikzcd}
\end{equation}

In what follows, we are interested in studying two closely related automorphisms of $\Dcat(B)$ and $\Dcat(A)$ known as the twist and cotwist around the restriction of scalars functor $F$.

\begin{definition} \label{twist def} Following \cite{An, AL}, let $S \colon \cA \to \cB$ be a functor between triangulated categories which admits a right adjoint $R$ and a left adjoint $L$. Then,

\begin{enumerate}
    \item The \textit{twist} around $S$ is a functor $T$ which sits in a functorial triangle
    \begin{align} \label{twist triang in D}
       SR \tow{\epsilon^\rmR} 1_{\cB} \tow{\alpha_1} T \tow{\alpha_2}^+  
    \end{align}
   where $\epsilon^\rmR \colon SR \to 1_{\cB}$ is the counit of the adjunction $(S, R)$.
        \item The \textit{cotwist} around $S$ is a functor $C$  which sits in a functorial triangle
    \begin{align} \label{cotwist triang in D}
       C \tow{\beta_1} 1_{\cA} \tow{\eta^\rmR} RS \tow{\beta_2}^+  
    \end{align}
     where $\eta^\rmR \colon 1_{\cA} \to RS$ is the unit of the adjunction $(S, R)$.
\end{enumerate}
\end{definition}

\begin{remark} \label{funct cones}
    Definition \ref{twist def} is adapted from the definition in \cite{AL}. Since taking cones in triangulated categories is not functorial, it is not clear, without any additional structure on the triangulated categories $\cA$ and $\cB$, whether the functors $T$ and $C$ exist or whether they are uniquely defined (even up to isomorphism). However, this paper concerns bimodules and so taking cones is sensible in our context. In the remainder of this remark, we will specify more precisely why "taking cones is sensible in our context". Mostly, this discussion is for the reader who would like a more rigorous treatment of twists around a functors.
    
    Let $\cA$ be the DG-category with one object $x$ and with morphisms defined by $\cA(x, x) \colonequals A$, with $A$  seen as a DG-algebra concentrated in degree zero. Then, a DG-module over $\cA$ corresponds to a chain complex of $A$-modules, and so we can easily check that $\cA$ is a Morita enhancement of $\Dcat(\Mod A)$. We may define a similar enhancement for $B$ and $A \otimes_\Z B^\op$. 

    Next, consider the DG-bimodules $s \colonequals {}_B B {}_A$, $r = \Rderived{\Hom_A}({}_B B, {}_A A)$ and $l = {}_A B {}_B$. Then, the underlying exact functors are $F$, $F^{\RA}$, and $F^{\LA}$ which induce the standard restriction-extension of scalars adjunction \eqref{res-ext diagram}. 
    
    Following the definitions in \cite[\S 5]{AL}, the twist and cotwist around $s$ are the DG-bimodules
\begin{align*}
    &  t = \cone(\epsilon^\mathrm{R}_A \colon \Rderived{\Hom_A}({}_B B, {}_A B) \Dtensor_B {}_B B {}_A \to {}_A A {}_A ) \\
    &  c = \cone(\eta^{\mathrm{R}}_B \colon {}_B B {}_B \to\Rderived{\Hom_A}({}_B B, {}_B B \Dtensor_B B {}_A ) )[-1]
\end{align*}
where $\epsilon^\mathrm{R}_A$ and $\eta^{\mathrm{R}}_B$ are the counit and unit of the adjunction $(-) \Dtensor_B {}_B B {}_A \dashv \Rderived{\Hom_A}({}_B B, -)$ evaluated at the bimodules ${}_A A {}_A$ and ${}_B B {}_B$, respectively. Thus, the twist and cotwist around $F$ are $ T =  (-) \Dtensor_A t$ and $ C =  (-) \Dtensor_B c$, respectively. 
\end{remark}

\bigskip
The intuition that the twist and cotwist are closely related can be made precise by the "2 out of 4" lemma. 

\begin{lemma}[{\cite[5.1]{AL}}] \label{spherical criteria}
    Let $S \colon \Dcat(\Mod B) \to \Dcat(\Mod A)$ be as in \ref{twist def}. If any two of the following conditions hold, then the other two also hold. 
\begin{enumerate}
    \item The twist $T \colon \Dcat(\Mod A) \to \Dcat(\Mod A)$ is an equivalence,
    \item The cotwist $C \colon \Dcat(\Mod B) \to \Dcat(\Mod B)$ is an equivalence,
    \item With $\alpha_2$ as in \eqref{twist triang in D}, there is a natural isomorphism
    \[LT[-1] \tow{L(\alpha_2[-1])} LSR \tow{\epsilon^\rmR_R} R \]
    \item \label{iso c relates adjoints} With $\beta_2$ as in \eqref{cotwist triang in D} and letting $\eta^\rmL$ denote the unit of the adjunction $(L, S)$, there is a natural isomorphism
    \[ R \tow{\eta^\rmL_R} RSL \tow{(\beta_2)_L} CL[1] \]
\end{enumerate}
\end{lemma}

\begin{definition}[{\cite[1.1]{AL}}] \label{spherical def}
    A functor $S \colon \Dcat(\Mod B) \to \Dcat(\Mod A)$ as above is \textit{spherical} if any two of the conditions in \ref{spherical criteria} hold. When $S$ is spherical, $T$ is called a \textit{spherical twist}.
\end{definition}

In practice, it is hard to check whether the specific natural transformation in \ref{spherical criteria} \eqref{iso c relates adjoints} is an isomorphism, and so the following result will be key. 

\begin{theorem}[{\cite[1]{add}, see also \cite{An} and \cite{Rou}}] \label{sph criterion add}
       With notation as above, if the cotwist $C \colon \Dcat(\Mod B) \to \Dcat(\Mod B)$ is an equivalence, and if there exists an isomorphism of functors $R \cong CL[1]$, then $S$ is spherical. 
\end{theorem}

\subsubsection{Specifying the twist and cotwist}
The twist and cotwist around $F$ are characterised in general in \cite{Go}. 

\begin{theorem}[{\cite[3.8,3.9]{Go}}] \label{twist for epi}
    The twist around $F$ is
    \begin{align*}
        \cT & = \Rderived{\Hom}_A({}_A \cone(p)[-1] {}_A, -) \colon \Dcat(A) \to \Dcat(A).
    \end{align*}
    Hence if $p$ is a surjection, then $\cT = \Rderived{\Hom}_A({}_A \ker p {}_A, -) $. 
\end{theorem}

In order to specify the cotwist, consider a K-projective resolution $Q \to B$ of the bimodule $B \in \Dcat(A \otimes_\Z B^\op)$. Then, it is straightforward to check that the quasi-isomorphism $Q \to B$ induces a quasi-isomorphism
 \[ s \colon {}_B B \Dtensor_A B {}_B \to \tot({}_B Q \otimes_A B {}_B ).\]
 In fact, we may take the resolution $q \colon Q \to {}_B B {}_A$ to be a chain map $(q^j)$ which vanishes in all degrees but zero. In which case, there is a chain map
\[ \beta = (\beta^j) \colon \tot({}_B Q \otimes_A B {}_B) \to {}_B B {}_B \] 
where $\beta^j = 0$ for $j \neq 0$ and $\beta^0$ is induced by composing $q^0 \otimes_A 1_B \colon Q^0 \otimes_A B {}_B \to {}_B B \otimes_A B {}_B$ with the multiplication isomorphism $m \colon  {}_B B \otimes_A B {}_B \to  {}_B B {}_B$. We may thus consider the composition 
\begin{align} \label{beta remark}
     s \cdot \beta \colon {}_B B \Dtensor_A B {}_B \to {}_B B {}_B. 
\end{align}

\begin{theorem} [{\cite[3.15]{Go}}] \label{cotwist gen}
    Let $C(s \cdot \beta)$ denote the cone of $s \cdot \beta \colon {}_B B \Dtensor_A B {}_B \to {}_B B {}_B$ in $\Dcat(B \otimes_\Z B^\op)$. Then, the functor
   \[ \cC = \Rderived{\Hom_B}( {}_B C(s \cdot \beta) {}_B, -) \]
   is the cotwist around $F$.
\end{theorem}

\section{Spherical twists for Gorenstein orders} \label{symmetric alg}

In this section, we will consider a surjection $p \colon A \to B$ with additional natural assumptions on $A$ and $B$ under which we can prove that the restriction of scalars functor $F$ is spherical. 

\subsection{Setting}

Motivated by geometric settings, we will work within the following setup. 

\begin{setup} \label[setup]{module fin}
    Let $\cR$ be a noetherian local Cohen-Macaulay (CM) ring of Krull dimension $d$ with canonical module $\omega_\cR$ and residue field $K$. Assume that $A$ is a module finite $\cR$-algebra, so the surjection $p \colon A \to B$ implies that $B$ is also module finite over $\cR$.
\end{setup}

   As the algebra $A$ is module finite over $\cR$, its derived category is equipped with the Nakayama functor defined in \cite[\S 3]{IR}, namely 
    \[ - \Dtensor_A \omega_A \colon \Dcat^{-}(\modu A) \to \Dcat^{-}(\modu A) \] 
    where $\omega_A \colonequals \Rderived{\Hom}_\cR(\Lambda, \omega_\cR)$ is a dualizing complex for $\Dcat(\modu A)$ in the sense of \cite{Ye}. Moreover, \cite[3.5,3.7]{IR} proves a form of Serre duality by showing the existence of the following functorial isomorphisms.
    \begin{enumerate} 
        \item For any $a \in \Db(\modu A)$ and $b \in \Kb(\proj A)$
        \begin{align} \label{derived serre duality} 
             \Rderived{\Hom}_A(a, b \Dtensor_A \omega_A) \cong \Rderived{\Hom}_\cR(\Rderived{\Hom}_A(b, a), \omega_\cR).
        \end{align}
        \item For any $a \in \Db(\fl A)$ and $b \in \Kb(\proj A)$
        \begin{align} \label{serre duality} 
            \Hom_{\Dcat(\modu A)}(a, b \Dtensor_A \omega_A [d]) \cong D  \Hom_{\Dcat(\modu A)}(b, a) .
        \end{align}
    \end{enumerate}

    \begin{remark}
        Although \cite[\S 3]{IR} states the above results for Gorenstein rings, their proof applies word for word to any CM ring equipped with a canonical module once $\cR$ is replaced by $\omega_\cR$.
    \end{remark}
    
    Recall that an algebra $A$ is called a \textit{Gorenstein $\cR$-order} if there is an $A$-bimodule isomorphism $\omega_A = \Rderived{\Hom_\cR}(A, \omega_\cR) \cong A$. Moreover, an algebra is self-injective if it is injective as a module over itself (see \cref{app: self inj} for a brief overview on self-injective algebras and the Nakayama permutation). 

    \begin{setup} \label[setup]{sym inj set up}
        Under \cref{module fin} assume in addition that $A$ is a Gorenstein $\cR$-order and that $B$ is a finite dimensional $K$-algebra which is either self-injective or has finite global dimension. Finally, assume that $B \in \Kcat^b(\proj A)$.
    \end{setup}

%    \begin{remark} \label{inj implies fin dim}
%        The key point of \ref{sym inj set up}, is that since $B$ is module finite over a noetherian ring and self-injective, \ref{GN result} implies that $B$ is finite dimensional. 
%    \end{remark}

    Recall that the dualizing complex of a finite dimensional algebra $B$ is $DB := \Hom_K(B, K)$ of $B$ over $K$. Hence, under set up \ref{sym inj set up}, the Nakayama functor of $B$ can be written as 
    \begin{align*}
        \cN & \colonequals (-) \Dtensor_B DB {}_B \colon \Db(\modu B) \to \Db(\modu B)
    \end{align*}
     and is known to be an exact equivalence whenever $B$ is self-injective or has finite global dimension. 
   
    In what follows, we are also interested in the the functors
\begin{align*}
    (-)^\dagger & \colonequals \Rderived{\Hom_\cR}(-, \omega_\cR) \colon \Dcat(\modu \cR) \to \Dcat(\modu \cR), \\
    D & \colonequals \Hom_\cR(-, k) \colon \Db(\fl \cR) \to \Db(\fl \cR). 
\end{align*}
    Since $\omega_\cR$ is a dualizing module, note that $(-)^\dagger$ induces a duality on $\Db(\modu \cR)$ so that $(-)^{\dagger \dagger} \cong 1_\cR$. 

\begin{lemma} \label{tor ext duality}
        Under \cref{sym inj set up}, suppose that $M$ and $N$ are $B$-bimodules with $M \in \Db(\fl \cR)$ and $N \in \Kb(\proj A)$. Then, there are $B$-bimodule isomorphisms
        \begin{align*}
            & {}_{B} \Tor^A_t(N, M) {}_{B}  \cong {}_{B} \Ext_A^{d-t}(D M, N) {}_{B} 
        \end{align*}
\end{lemma}
\begin{proof}
    Consider the following isomorphisms,  
    \begin{align*}
        \Rderived{\Hom_A}({}_{B} M {}_{A} {}^\dagger , {}_{B} N {}_{A}) & \cong \Rderived{\Hom_A}({}_{B} M {}_{A} {}^\dagger, {}_{B} N {}_{A} \Dtensor  {}_{A} (\omega_A) {}_A  ) \tag{by assumptions \ref{sym inj set up}} \\
        & \cong \Rderived{\Hom_A}({}_{B} N {}_{A}, {}_{B} M {}_{A}  {}^\dagger )^\dagger  \tag{by Serre duality \eqref{derived serre duality}}
    \end{align*}
    Moreover, it follows from the derived tensor-hom adjunction that
    \[ \Rderived{\Hom_A}({}_{B} N {}_{A}, {}_{B} \Rderived{\Hom_\cR}(M, \omega_\cR) {}_{A} ) \cong \Rderived{\Hom_\cR}({}_{B} N {}_{A} \Dtensor {}_{A} M {}_{B}, \omega_\cR ). \]
    Hence, 
    \begin{align*}
        \Rderived{\Hom_A}({}_{B} N {}_{A}, {}_{B} M {}_{A} {}^\dagger )^\dagger & \cong ({}_{B} N {}_{A} \Dtensor {}_{A} M {}_{B})^{\dagger \dagger} \\
        & \cong {}_{B} N {}_{A} \Dtensor {}_{A} M {}_{B}. \tag{since $(-)^\dagger$ is a duality}
    \end{align*}
    Thus, there is are bimodule isomorphisms
    \begin{align*}
         {}_{B} \Ext_A^{-t}(M^\dagger, N) {}_{B} & = \cohom{-t}( \Rderived{\Hom_A}({}_{B} M^\dagger {}_{A}, {}_{B} N {}_{A}))  \\
         & \cong \cohom{-t}({}_{B} N {}_{A} \Dtensor {}_{A} M {}_{B}) \\
         & = {}_{B} \Tor^A_t(N, M) {}_{B}. 
    \end{align*}

    Lemma \cite[3.6]{IR} states that there is a functorial isomorphism $(-)^\dagger \cong [-d] \cdot D$ on $\Db(\fl \cR)$. Hence, 
    \[ {}_{B} \Tor^A_t(N, M) {}_{B}  \cong  {}_{B} \Ext_A^{-t}(N^\dagger, M) {}_{B} \cong  {}_{B} \Ext_A^{d-t}(D M, N) {}_{B}. \qedhere\]
\end{proof}

\begin{cor} \label{CMness}
     Under \cref{sym inj set up}, then $\cohom{k}( {}_B B \Dtensor_A DB {}_B) = 0$ for all $k \neq 0, -d$ if and only if $\cohom{k}(\Rderived{\Hom_A}(B, B)) = 0$ for all $k \neq 0, d$. 
\end{cor}
\begin{proof}
    Suppose that $M$ and $N$ are $B$-bimodules with $M \in \Db(\fl \cR)$ and $N \in \Kb(\proj A)$. From the proof of \ref{tor ext duality}, there is an isomorphism
        \begin{align*}
       \Rderived{\Hom_A}({}_{B} M {}_{A} {}^\dagger , {}_{B} N {}_{A}) & \cong {}_{B} N {}_{A} \Dtensor {}_{A} M {}_{B}.
    \end{align*}
    Thus, by letting $N = B$ and $M = B^\dagger$, it follows that 
    \begin{align*}
         \Rderived{\Hom_A}({}_{B} B {}_{A}, {}_{B} B {}_{A}) & \cong \Rderived{\Hom_A}({}_{B} B {}_{A} {}^{\dagger \dagger} , {}_{B} B {}_{A}) \tag{since $(-)^\dagger$ is a duality}
         \\ 
         & \cong {}_{B} B {}_{A} \Dtensor {}_{A} B {}_{B} {}^\dagger \\
         & \cong {}_{B} B {}_{A} \Dtensor {}_{A} DB {}_{B} [-d] \tag{Since $(-)^\dagger \cong [-d] \cdot D$ on $\Db(\fl \cR)$ by \cite[3.6]{IR}}
    \end{align*}
    Thus, the statement follows.
\end{proof}

\subsection{Spherical twists for Gorenstein orders}

Having introduced the setting of this section and its main properties in the previous subsection, we construct spherical twists for Gorenstein orders in this subsection. 

Our first goal is to show that the cotwist is a shift of the Nakayama functor of $B$. For this, we will need a few lemmas. 

\begin{lemma} \label{inverse bimod}
    Under \cref{sym inj set up}, suppose that $\Ext^k_A(B, B) = 0$ for all $k \neq 0, d$. Then,  there is an isomorphism ${}_B C(s \cdot \beta) {}_B \cong \Rderived{\Hom_B}({}_B DB {}_B, B)[d+1]$.
\end{lemma}
\begin{proof}
    Under the assumptions \cref{sym inj set up}, the functor $\cN \colonequals (-) \Dtensor_B DB {}_B$ is an equivalence, and so it follows from \cite[4.1, 4.2]{Ric} that the bimodule ${}_B DB {}_B$ admits an inverse $B$-bimodule complex $M := \Rderived{\Hom_B}(DB, B)$. By inverse bimodule complex, we mean a complex $M$ such that there are isomorphisms $M \Dtensor_B DB \cong DB \Dtensor_B M \cong B$. Note, moreover, that the inverse bimodule complex is unique up to isomorphism. To see this, suppose that $M, N$ are bimodule inverses of $DB$. Then,
    \[ M \cong M \Dtensor_B DB \Dtensor_B N \cong N. \]
    We will show that ${}_B C(s \cdot \beta) {}_B$ is a shift of the inverse bimodule of $DB$.
    
    Applying the derived Nakayama functor to the triangle defining $C(s \cdot \beta)$, induces a triangle
    \[ {}_B B \Dtensor_A B \Dtensor_B DB {}_B \tow{ s \cdot \beta \Dtensor_B DB} {}_B B \Dtensor_B DB {}_B \to {}_B C(s \cdot \beta) \Dtensor_B DB {}_B \to^+. \]
    For convenience, let $Y := {}_B C(s \cdot \beta) \Dtensor_B DB {}_B$ and $f \colonequals s \cdot \beta \Dtensor_B DB$. Moreover, observe that there is a natural multiplication map $m \colon {}_B B \otimes_A DB {}_B \to {}_B DB {}_B$. 
    
    Then, $Y \cong \cone(f)$ where, by fixing a projective resolution $q \colon Q \to B$ in $\Dcat(A \otimes_\Z B^\op)$ , we may write $f$ as the chain map
    {\small
    \[
    \begin{tikzcd}
        \cdots \arrow[r] \arrow[d] & Q^i \otimes_A B \otimes_B DB \arrow[r, "d^i"] \arrow[d] & Q^{i+1} \otimes_A B \otimes_B DB \arrow[r, "d^{i+1}"] \arrow[d] & \cdots \arrow[r, "d^{-1}"] \arrow[d] & Q^0 \otimes_A B \otimes_B DB \arrow[r] \arrow[d, "{(s \cdot \beta)^{0} \otimes_B DB }", two heads] & 0 \arrow[d]  \\
         \cdots \arrow[r] & 0 \arrow[r] \arrow[d] & 0 \arrow[r] \arrow[d] & \cdots \arrow[r] \arrow[d] & B \otimes_B DB
         %\cohom{0}({}_B B \Dtensor_A DB {}_B) 
         \arrow[d, "\sim"{anchor=south, rotate=90}, "m"] \arrow[r] & 0 \arrow[d]  \\
         \cdots \arrow[r] & 0 \arrow[r] & 0 \arrow[r] & \cdots \arrow[r] & DB \arrow[r] & 0.  \\
    \end{tikzcd}
    \]
    }
    It is straightforward to check that $(s \cdot \beta)^{0} \otimes_A DB=:f_0$ induces an isomorphism on cohomologies of the complexes. In particular, this means that $\Img d^{-1} \cong \ker f_0$ by the snake lemma. Therefore, $Y \cong \cone(f)$ is isomorphic to the chain complex
   \[
    \begin{tikzcd}[ampersand replacement=\&]
        \cdots \arrow[r] \& Q^i \otimes_A DB \arrow[r, "d^i"] \& Q^{i+1} \otimes_A DB \arrow[r, "d^{i+1}"] \& \cdots \arrow[r, "d^{-1}"] \& Q^0 \otimes_A DB \arrow[r, "f_0"] \& DB \arrow[r] \& 0 
    \end{tikzcd}
    \]
    which is exact in degrees $0$ and $-1$. Thus, there is a quasi-isomorphism $Y \cong \tau_{\leq -2}(B \Dtensor_A DB)[-1]$, where $\tau_{\leq n}$ denotes the canonical truncation functor. 

    If $\cohom{k}(\Rderived{\Hom_A}(B, B)) = 0$ for all $k \neq 0, d$, then by \ref{CMness}, $\cohom{k}( {}_B B \Dtensor_A DB {}_B) = 0$ for all $k \neq 0, -d$. It follows that $Y \cong \tau_{\leq -2}(B \Dtensor_A DB)[-1]$ has cohomology concentrated in degree $-d-1$. Thus, and by the duality \ref{tor ext duality},
    \[ Y \cong \cohom{-d}(B \Dtensor_A DB)[-d-1] = \Tor_d^A(B, DB) \cong B[-d-1] \]

    In other words, ${}_B C(s \cdot \beta) [d + 1] \Dtensor_B DB {}_B \cong B$. That is, ${}_B C(s \cdot \beta) {}_B$ is a shift of the inverse bimodule complex of ${}_B DB {}_B$. Thus, ${}_B C(s \cdot \beta) {}_B \cong \Rderived{\Hom_B}( {}_B DB {}_B, B)[d+1]$, as required. 
\end{proof}

When $B$ is self-injective, then we can further simplify the characterisation of the bimodule $C(s \cdot \beta)$.

\begin{lemma} \label{ctw res cohom in 2 degrees}
    Suppose that $\cohom{k}( {}_B B \Dtensor_A DB {}_B) = 0$ for all $k \neq 0, -d$. If $B$ is self-injective then there is a $B$-bimodule isomorphism
    \[ {}_B C(s \cdot \beta) {}_B \cong  \Tor_{d}^A(B, B).  \]
\end{lemma}
\begin{proof}
    This is essentially \cite[3.16]{Go}, but we present a proof here for convenience. 

    We first claim that $\cohom{k}(B \Dtensor_A B) = 0$ for all $k \neq 0, -d$. Well, since $B$ is finite dimensional, let us write $B = \bigoplus_{i = 1}^n P_i$ as a finite direct sum of its indecomposable projectives. Then, since $B$ is self-injective, $DB = \bigoplus_{i = 1}^n P_{\sigma^{-1}(i)}$, where $\sigma$ is the Nakayama permutation of $B$. We therefore have 
    \[\cohom{k}(B \Dtensor_A DB) = \bigoplus_{i = 1}^n \cohom{k}(B \Dtensor_A P_{\sigma^{-1}(i)}) = 0 \]
    for all $k \neq 0, -d$. Thus, necessarily $\cohom{k}(B \Dtensor_A P_{\sigma^{-1}(i)}) = 0$ for all $i$ and $k \neq 0, -d$. Since $\sigma$ is a permutation, we may conclude that  $\cohom{k}(B \Dtensor_A P_{i}) = 0$ for all $i$. Thus, 
    \[ \cohom{k}(B \Dtensor_A B) = \bigoplus_{i = 1}^n \cohom{k}(B \Dtensor_A P_i) = 0 \]
    as required.
    
    Next, note that the triangle of $B$-bimodules
    \[ {}_B B \Dtensor_A B {}_B \tow{s \cdot \beta} {}_B B {}_B \to {}_B C(s \cdot \beta) {}_B \to^+ \]
    induces the long exact sequence
    \[ \hdots \to \cohom{k}(B \Dtensor_A B) \tow{\cohom{k}(s \cdot \beta)} \cohom{k}(B) \to \cohom{k}(C(s \cdot \beta)) \to \cohom{k+1}(B \Dtensor_A B) \to \hdots \]
    on cohomology. Since $B$ is concentrated in degree zero and $\cohom{k}( {}_B B \Dtensor_A B {}_B)$ vanishes in all degrees $k \neq 0, -d$, it must be that $\cohom{k}(C(s \cdot \beta)) = 0$ for all $k \neq -1, 0, -d-1$. Moreover, note that $\cohom{0}(s \cdot \beta)$ is an isomorphism. Hence,
    \[ \cohom{-1}(C(s \cdot \beta)) = \cohom{0}(C(s \cdot \beta)) = 0. \]
    Therefore, $C(s \cdot \beta)$ is concentrated in degree $t-1$. It  follows from this and the long exact cohomology sequence that
    \[ C(s \cdot \beta) \cong \cohom{-d-1}(C(s \cdot \beta))[t-1] \cong \cohom{-d}(B \Dtensor_A B)[t-1]. \]
    Thus, the statement holds.
\end{proof}

As a result, we can at least establish that as a complex of $B$-modules then $C(s \cdot \beta)$ is bounded .

\begin{lemma} \label{C is bounded}
    The complex $C(s \cdot \beta)$ is bounded as a complex in $\Dcat(\modu B)$ and $\Dcat(\modu B^\op)$.
\end{lemma}
\begin{proof}
    If $B$ is self-injective, then $C(s \cdot \beta)$ is concentrated in a single degree by \ref{ctw res cohom in 2 degrees} and the result follows. 

    If on the other hand $B$ has finite global dimension then $DB$ admits a finite projective resolution as a left or right $B$-module. It is thus clear that $C(s \cdot \beta) \cong \Rderived{\Hom_B}(DB, B)[d+1]$ is bounded.
\end{proof}

With these lemmas in hand, we are finally able to specify the cotwist. 

\begin{prop} \label{finite cotwist}
      Under \cref{sym inj set up}, suppose that  $\Ext^k_A(B, B) = 0$ for all $k \neq 0, d$. Then, the cotwist around $F$ restricted to is $\cC = [-d-1] \cdot \cN$.
\end{prop}
\begin{proof}
    It follows from \ref{cotwist gen} that the cotwist $C$ around $F$ is
    \[ \cC = \Rderived{\Hom_B}( {}_B C(s \cdot \beta) {}_B, -). \]
    The left adjoint functor of $\cC$ is, therefore, $\cC^* \colonequals (-) \Dtensor_B C(s \cdot \beta) {}_B$. Moreover, it follows from \ref{inverse bimod} that $\cC^* \cong [d+1]
     \cdot \cN^{-1}$, which is an equivalence under the assumptions of \cref{sym inj set up}. Thus, $\cC$ is the inverse of $\cC^*$, meaning that $\cC \cong [-d-1] \cdot \cN$.
\end{proof}

Consequently, the restriction of scalars functor is spherical. 

\begin{theorem}\label{spherical condition 4}
     Under \cref{sym inj set up}, then $F^\RA \cong [1] \cdot \cC \cdot F^\LA$ on $\Db(A)$. In particular, both $F^\RA$ and $F^\LA$ preserve bounded complexes. Hence, if  $\Ext^k_A(B, B) = 0$ for all $k \neq 0, d$, then $F \colon \Db(\modu B) \to \Db(\modu A)$ is spherical. 
\end{theorem} 
\begin{proof}
    Let $c \in \Db(A)$, Then, there are the following functorial isomorphisms
    \begin{align*}
        F^\RA(c) & = \Rderived{\Hom_A}({}_B B, c) \\
                 & \cong \Rderived{\Hom_A}({}_B B, c \Dtensor_A A {}_A) \\
                 & \cong c \Dtensor_A \Rderived{\Hom_A}({}_B B, A ) \tag{by \cite[2.10(2)]{IR}, since $B$ is perfect}\\
                 & \cong c \Dtensor_A \Rderived{\Hom_A}({}_B B, A^\dagger ) \tag{since $A$ is a Gorenstein $\cR$-order } \\
                 & \cong c \Dtensor_A ( {}_B B {}_A)^\dagger \tag{by derived tensor-hom adjunction} \\
                 & \cong c \Dtensor_A D B {}_B [-d] \tag{by \cite[3.6]{IR} since $B$ is finite dimensional} \\
                 & \cong c \Dtensor_A B {}_B \Dtensor_B DB {}_B [-d] \\
                 & \cong [1] \cdot \cC \cdot F^\LA(c) \tag{by \ref{finite cotwist} since $\cC = (-) \Dtensor_B DB [-d-1]$.}
    \end{align*}

    Since $B$ is perfect as a right $A$-module, then $F^\LA$ preserves bounded complexes. Moreover $\cC$ also preserves bounded complexes, and so necessarily $F^\RA$ preserves bounded complexes. 
    
    Next, suppose that  $\Ext^k_A(B, B) = 0$ for all $k \neq 0, d$. Then, the fact that $F$ is spherical follows immediately from the above combined with \ref{finite cotwist} and addendum to the $2$ out of $4$ property \ref{sph criterion add}.
\end{proof}

Since under the assumptions of \cref{sym inj set up}, $B$ is finite dimensional, it follows that $B$ has finitely many simples up to isomorphism. Thus, let $\Omega = \{ S_j \}_{j=1}^k$ denote the set containing one representative from each isomorphism class of simple modules. If $B$ is self-injective, let $\sigma$ be the Nakayama permutation of $B$.

\begin{cor} \label{permutes simples}
Under \cref{sym inj set up}, suppose that $\Ext^k_A(B, B) = 0$ for all $k \neq 0, d$. Then, the diagram
    \[ \begin{tikzcd}[column sep = 5em]
        \Db(\modu A) \arrow[r, "\cT"] & \Db(\modu A) \\
        \Db(\modu B) \arrow[u, "F"] \arrow[r, "{\cN(-)[-d+1]}"] & \Db(\modu B) \arrow[u, "F"]
    \end{tikzcd} \]
    commutes. In particular, 
    \begin{align*}
        & \Rderived{\Hom_A}(\ker p, {}_{B} B) \cong {}_{B} D B {}_A [-d + 1].
    \end{align*}
    If, further, $B$ is self-injective then
    \begin{align*}
        & \Rderived{\Hom_A}(\ker p, S_i) \cong S_{\sigma^{-1}(i)} {}_A [-d + 1].
    \end{align*}
\end{cor}
\begin{proof}
    By \ref{spherical condition 4}, $F$ is spherical, so it follows from \cite[section 1.3]{add} and \ref{finite cotwist} that there are isomorphisms $\cT \cdot F(-) \cong F \cdot \cC(-)[2] \cong F \cdot \cN(-)[-d+1]$.
\end{proof}

\subsection{Spherical twists in dimensions $2$ and $3$} 
The Ext-vanishing condition required to construct the spherical twist may seem difficult to compute in general. However, when $\dim \cR = 2, 3$, this condition can be further simplified. 

\begin{lemma} \label{tor vanishing}
    Under \cref{sym inj set up}, suppose that $\dim \cR = 2, 3$ and $(\ker p)^2 \cong \ker p$. Then, for all $k \neq 0, 3$, $\Ext^k_A(B, B) = 0$. If moreover, $B$ is self-injective, then the converse is true.
\end{lemma}
\begin{proof}
    Consider the case $\dim \cR = 2$. Since $B \in \Kcat^b(\proj A)$, we may apply the Auslander-Buchsbaum formula for Gorenstein orders \cite[2.16]{IW1}
    \[ \pdim_A B \leq \pdim_A B + \depth_\cR B = 2. \]
    Whence, $\Ext^k_A(B, B) = 0$ for $k < 0$ and $k > 2$. Thus, it suffices to show that $\Ext^1_A(B, B)$ vanishes.

    Consider next the case $\dim \cR = 3$. Then, as before, we may apply the Auslander-Buchsbaum \cite[2.16]{IW1}
    \[ \pdim_A B \leq \pdim_A B + \depth_\cR B = 3 \]
    and so $\Ext^k_A(B, B) = 0$ for $k < 0$ and $k > 3$. Moreover, by Serre duality \eqref{serre duality}, there is a vector space isomorphism $\Ext^2_A(B, B) \cong D \Ext^1_A(B, B)$. Thus, again, it suffices to show that $\Ext^1_A(B, B)$ vanishes. 
    
    Well, applying the functor $\Hom_A(-, B)$ to the exact sequence
    \begin{align} \label{kerp seq}
        0 \to \ker p \to A \to  A/ \ker p \ (\cong B) \to 0 
    \end{align}
induces that long exact sequence
    \begin{align*}
        0 \to \Hom_A(B, B) \to \Hom_A(A, B) \to \Hom_A(\ker p, B) \to \Ext^1_A(B, B) \to 0.
    \end{align*}
    However, $\Hom_A(\ker p, B) = 0$. To see this, consider a morphism $f \colon \ker p \to A/ \ker p$. Since $\ker p = (\ker p)^2$, then any $x \in \ker p$ is such that $x = x_1 x_2$ for $x_1, x_2 \in \ker p$. Thus, necessarily $f(x) = f(x_1 x_2) = f(x_1) x_2$. Since $x_2 \in \ker p$, it must be that $f(x_1) x_2 = 0$. Thus, $f$ is the zero map and we may conclude that $\Hom_A(\ker p, B) = 0$, which implies that $\Ext^1_A(B, B) = 0$.

    Suppose next that $B$ is self-injective and that $\Ext^1_A(B, B) = 0$. We will prove the statement for the case $\dim \cR = 3$, but the proof for the two-dimensional case is very similar. 
    
    Since $B \cong DB$ by \ref{tor ext duality}, there are vector space isomorphisms   
    \[ \Tor_1^A(B, B) \cong \Tor_1^A(B, DB) \cong \Ext^2_A(B, B) \cong D \Ext^1_A(B, B) = 0 \]
   It is a classical result that $\Tor_1^A(B, B) = \Tor_1^A(A/\ker p, A/\ker p) \cong \ker p / (\ker p)^2$. Indeed, applying the functor $- \otimes_A A/ \ker p$ to the exact sequence \eqref{kerp seq} yields the exact sequence
    \[ 0 \to \Tor_1^A(B, B) \to \ker p \otimes_A A/\ker p \tow{f} A \otimes_A A/\ker p \to A/ \ker p \otimes_A A/\ker p \to 0.\]
    It is easy to check that $f = 0$ and, thus, $ \Tor_1^A(B, B) \cong \ker p \otimes_A A/\ker p$. Moreover, applying the functor $- \otimes \ker p$ to \eqref{kerp seq} induces an exact sequence
    \[  \ker p \otimes_A \ker p \tow{g} A \otimes_A \ker p \to A/\ker p  \otimes_A \ker p \to 0.\]
    It is straightforward that $\Img g \cong (\ker p)^2$. We may conclude that  
    \[ \Tor_1^A(B, B) \cong A/\ker p  \otimes_A \ker p \cong \ker p / (\ker p)^2\]
    and thus the statement follows.
\end{proof}

\begin{cor} \label{3d cotwist}
    Under \cref{sym inj set up}, suppose that $\dim \cR = 2, 3$. If $(\ker p)^2 \cong \ker p$, then, the cotwist around $F$ is $\cC \cong [-d-1] \cdot \cN$, where $\cN(-)$ denotes the derived Nakayama autoequivalence of $B$.
\end{cor}
\begin{proof}
    In light of \ref{finite cotwist}, the key point to prove is that  $\Ext^k_A(B, B) = 0$ for all $k \neq 0, 3$. This is exactly \ref{tor vanishing}. 
\end{proof}

\begin{cor} \label{3d spherical}
    Under \cref{sym inj set up}, suppose that $\dim \cR = 2, 3$. If $(\ker p)^2 = \ker p$, then, the restriction of scalars functor $F \colon \Db(\modu B) \to \Db(\modu A)$ is spherical. Moreover, the diagram
    \[ \begin{tikzcd}[column sep = 5em]
        \Db(\modu A) \arrow[r, "\cT"] & \Db(\modu A) \\
        \Db(\modu B) \arrow[u, "F"] \arrow[r, "{\cN(-)[-2]}"] & \Db(\modu B) \arrow[u, "F"]
    \end{tikzcd} \]
    commutes. In particular, 
    \begin{align*}
        & \Rderived{\Hom_A}(\ker p, {}_{B} B) \cong {}_{B} D B {}_A [-2].
    \end{align*}
    If, moreover, $B$ is self-injective then
    \begin{align*}
        & \Rderived{\Hom_A}(\ker p, S_i) \cong S_{\sigma^{-1}(i)} {}_A [-2].
    \end{align*}
\end{cor}
\begin{proof}
    This follows directly from \ref{spherical condition 4}, \ref{permutes simples} and \ref{3d cotwist}.
\end{proof}

\begin{example} \label{nc twist first ex}
    Let $R$ be a noetherian complete local commutative Gorenstein ring with at worst isolated hypersurface singularities and $\dim R = 3$. Moreover, consider $M \in \CM R$ a maximal Cohen-Macaulay module with $\Ext^1_R(M, M) =0$, so that $\End_R(M) \in \CM R$. Assume further that $R$ is a summand of $M$.
    
    Write $\Lambda = \End_R(M)$ and let $[\add R]$ be the ideal of $\Lambda$ consisting of maps $M \to M$ which factor through $\add R$. Set $\Lambda_\con = \Lambda/[\add R] = \uEnd_R(M)$, $\Lambda_\con = \Lambda/[\add R] = \uEnd_R(M)$. We will show that the natural surjection $\pi \colon \Lambda \to \Lambda_\con$ satisfies the assumptions of \ref{3d spherical}.
    
    Since CM modules are reflexive and $M \in \CM R$, observe that $\Lambda$ is a Gorenstein $R$-order. To see this, we refer to the results \cite[2.22(2)]{IW1} and \cite[3.8 (1)$\Rightarrow$(3)]{IR}. Since $\Lambda$ is module finite, it satisfies \cref{module fin}.  
    
    Next, we claim that $(\ker \pi)^2 = \ker \pi$. Well, $\ker \pi = [\add R] = \Lambda e_0 \Lambda$, where $e_0 \colon M \to R \hookrightarrow M$ is the idempotent corresponding to the projection onto the summand $R$ of $M$. It is easy to check that $(\Lambda e_0 \Lambda)^2 = \Lambda e_0 \Lambda$. Additionally, $\Lambda_\con \in \Kb(\proj \Lambda)$ by \cite[A.7(3)]{Wem01}. It remains to show that $\Lambda_\con$ is self-injective. This follows from \cite[7.1]{BIKR}.
    
     Whence, we may apply \ref{3d spherical}, and we may conclude that \ref{3d spherical} extends \cite[5.11,5.12]{DW1}. 
\end{example}

\section{Applications to skew group algebras and Hilbert schemes}

\subsection{Quotient singularities and skew group algebras.} \label{sec:mckay}

This section briefly recalls some aspects of the \textit{special McKay correspondence} in dimension $3$. This correspondence provides some technology to construct examples of varieties which are derived equivalent to Gorenstein orders.

\begin{setup} \label[setup]{sl mckay setup}
    Consider a finite abelian group $G \subset \SL(3, \C)$ acting on the $\C$-algebra $R = \C[x, y, z]$. Let $R^G$ denote the ring of invariants and set $\C^3/G \colonequals \spec R^G$. 
\end{setup}

\medskip
It turns out that $\spec R^G$ admits a crepant resolution (known as $G$-$\mathrm{Hilb}(\C^3)$) \cite{Nakamura} which, by a result of \cite{BKR}, is derived equivalent to the \textit{skew-group algebra}. 

\begin{definition}
    Recall that the \textit{skew-group} algebra $R\#G$ is the vector space $R \otimes_\C \C G$ equipped with a "skewed" multiplication rule. That is, for $r_1 \otimes g_1, r_2 \otimes g_2 \in R \otimes_\C \C G$, 
    \[ (r_1 \otimes g_1)(r_2 \otimes g_2) \colonequals (r_1 g_1(r_2)) \otimes g_1 g_2. \]
\end{definition}

In other words, there is a triangulated equivalence
\[ \Phi_G \colon  \Db(\modu R\#G) \tow{\sim} \Db(\coh G\mathrm{-Hilb}(\C^3)). \]
 Thus, autoequivalences $\Db(\modu R\#G)$ induce autoequivalences of $\Db(\coh G\mathrm{-Hilb}(\C^3))$. 
 
In order to compute examples, we are, therefore, interested in presenting $R\#G$ as a quiver with relations. This can be done via the \textit{Mckay quiver} \cite{Mck}, which can be constructed as follows. First, note that under \cref{sl mckay setup}, since $G$ is abelian, it is possible to choose a basis of $\C^3$ which diagonalises the action of $G$ on $\C^3$. Thus, we may assume that each $g \in G$ is a diagonal matrix.

\begin{notation} \label{mckay notation}
    Assume \cref{sl mckay setup}.
    For each $i = 1, 2, 3$, let $\chi_i \colon G \to \C$ denote the one-dimensional representation of $G$ defined by setting $\chi_i(g)$ to be the $i$th diagonal element of $g$ for each $g \in G$. Moreover, let $\{\rho_0, \rho_1, \dotsb, \rho_m\}$ be the complete set of irreducible $\C$-representations of $G$. 
\end{notation}

\begin{prop}[See e.g.\ {\cite{BSW}}]
    Under \cref{sl mckay setup} and with notation as in \ref{mckay notation}, the McKay quiver of $G$ can be described as the directed graph with vertices $\rho_i$ and arrows
    \begin{align*}
        &  \chi_1 \otimes \rho_k \tow{x_{\rho_k} }\rho_k, \\
        &  \chi_2 \otimes \rho_k \tow{y_{\rho_k} }\rho_k. \\
        &  \chi_3 \otimes \rho_k \tow{z_{\rho_k} }\rho_k. 
    \end{align*}
    for all $0 \leq k \leq m$.
\end{prop}

\begin{prop}[{\cite{IT,BSW, CMT, LW}}] \label{aus t}
    Let $G$ and $R$ be as in \cref{sl mckay setup}. Then, the following hold.
    \begin{enumerate}
        \item There is an isomorphism $\alpha \colon R\#G \to \End_{R^G}(R)$ sending $r_1 \otimes g_1 \in R\#G$ to the endomorphism $\mu_{r_1 \otimes g_1} \colon R \to R$ defined by the rule $\mu_{r_1 \otimes g_1}(r) = r_1 g_1(r)$ for all $r \in R$. \label{aus 1}
        \item \label{aus 2} \label{aus 3} $R\#G$ can be presented, up to a Morita equivalence, as the path algebra of the McKay quiver, subject to the relations 
        \[ \langle \{ x_{\rho_k} y_{\chi_1 \otimes \rho_k} - y_{\rho_k} x_{\chi_2 \otimes \rho_k}, x_{\rho_k} z_{\chi_1 \otimes \rho_k} - z_{\rho_k} x_{\chi_3 \otimes \rho_k}, y_{\rho_k} z_{\chi_2 \otimes \rho_k} - z_{\rho_k} y_{\chi_3 \otimes \rho_k} \, \forall \, 1 \leq k \leq m \}\rangle. \] 
    \end{enumerate}
\end{prop}
\begin{proof}
    Statement \eqref{aus 1} is \cite[5.15]{LW} (see also \cite[4.2]{IT}). Statement \eqref{aus 2} follows from \cite{CMT}, as well as \cite[3.2]{BSW} and \cite[4.1]{BSW}.
\end{proof}

\begin{example} \label{mckay example}
    Let $\epsilon_3$ be a primitive third root of unity, and let $G \subset \SL(3, \C)$ be the group 
    \[ G = \left \langle \begin{pmatrix} \epsilon_3 & 0 & 0 \\ 0 & \epsilon_3 & 0 \\ 0 & 0 & \epsilon_3 \end{pmatrix}\right \rangle \]
    acting on $R = \C[x,y,z]$ as $g * f(x, y, z)  = f\left( (x, y, z) g \right)$. Then, the McKay quiver of $G$ is 
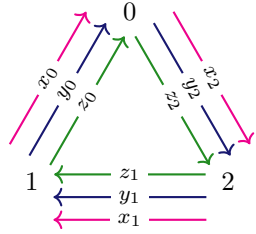
\vspace{-3em}
\[
\begin{tikzpicture}
\draw (-4,2) node [circle] {$Q:$};

\foreach \b [count=\a] in {0,...,2}{%
\draw (\a*360/3-30: 1.5cm) node [circle, text width=5mm, align=center] (\b) {\b};
}
\foreach \b [count=\a] in {3,...,5}{%
\draw (\a*360/3-30: 2.5cm) node [circle, text width=20.5mm, align=center] (\b) {};
}
\foreach \b [count=\a] in {6,...,8}{%
\draw (\a*360/3-30: 1.9cm) node [circle, text width=10mm, align=center] (\b) {};
}
\foreach \b [count=\a] in {9,...,11}{%
\draw (\a*360/3-30: 1.3cm) node (\b) {};
}

\draw [draw = magenta, ->, thick] (3) edge node[pos=0.5, fill=white, inner sep=1mm,sloped] {\footnotesize $x_{2}$} (5); 
\draw [draw = magenta, ->, thick] (5) edge node[pos=0.5, fill=white, inner sep=1mm,sloped] {\footnotesize $x_{1}$} (4);
\draw [draw = magenta, ->, thick] (4) edge node[pos=0.5, fill=white, inner sep=1mm,sloped] {\footnotesize $x_{0}$} (3); 

\draw [draw = MidnightBlue, ->, thick] (6) edge node[pos=0.5, fill=white, inner sep=1mm,sloped] {\footnotesize $y_{2}$} (8); 
\draw [draw = MidnightBlue, ->, thick] (8) edge node[pos=0.5, fill=white, inner sep=1mm,sloped] {\footnotesize $y_{1}$} (7);
\draw [draw = MidnightBlue, ->, thick] (7) edge node[pos=0.5, fill=white, inner sep=1mm,sloped] {\footnotesize $y_{0}$} (6); 

\draw [draw = ForestGreen, ->, thick] (9) edge node[pos=0.5, fill=white, inner sep=1mm,sloped] {\footnotesize $z_{2}$} (11); 
\draw [draw = ForestGreen, ->, thick] (11) edge node[pos=0.5, fill=white, inner sep=1mm,sloped] {\footnotesize $z_{1}$} (10);
\draw [draw = ForestGreen, ->, thick] (10) edge node[pos=0.5, fill=white, inner sep=1mm,sloped] {\footnotesize $z_{0}$} (9); 
\end{tikzpicture}
\]
\vspace{-3em}

Thus, The skew group algebra $R\#G$ can be presented as the path algebra of $Q$ over $\C$ subject to the relations
    \begin{align*} I = \langle & x_i y_{i-1} - y_i x_{i-1}, x_i z_{i-1} - z_i x_{i-1}, y_i z_{i-1} - z_i y_{i-1} \rangle_{i \in \Z_3}.
    \end{align*}
   These relations intuitively mean that the arrows corresponding to $x, y, z$ "commute". 
\end{example}

\begin{example} \label{mckay example 2}
    Let $\epsilon_7$ be a primitive seventh root of unity, and let $G' \subset \SL(3, \C)$ be the group 
    \[ G' = \left \langle \begin{pmatrix} \epsilon_7 & 0 & 0 \\ 0 & \epsilon_7 & 0 \\ 0 & 0 & \epsilon_7^5 \end{pmatrix}\right \rangle. \]
    acting on $R$ as $g * f(x, y, z)  = f\left( (x, y, z) g \right)$. Then, the McKay quiver of $G'$ is 
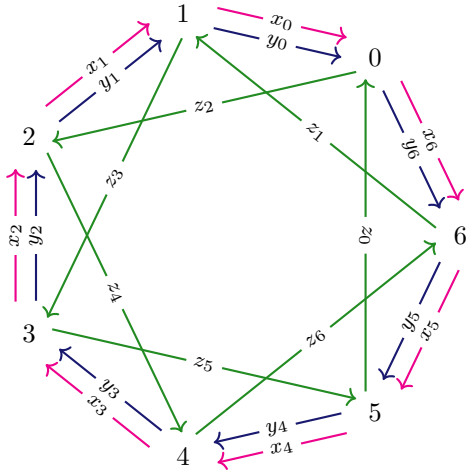
\[
\begin{tikzpicture}

\draw (-5,3) node [circle] {$Q':$};

\foreach \b [count=\a] in {0,...,6}{%
\draw (\a*360/7: 3cm) node [circle, text width=5mm, align=center] (\b) {\b};
}
\foreach \b [count=\a] in {7,...,13}{%
\draw (\a*360/7: 3.2cm) node [circle, text width=7.5mm, align=center] (\b) {};
}
\foreach \b [count=\a] in {14,...,20}{%
\draw (\a*360/7: 2.9cm) node [circle, text width=5mm, align=center] (\b) {};
}
\foreach \b [count=\a] in {21,...,27}{%
\draw (\a*360/7: 2.8cm) node (\b) {};
}

\foreach \i [count=\a] in {8,...,13}{%
\draw [draw = magenta, ->, thick] (\i) edge node[pos=0.5, fill=white, inner sep=1mm,sloped] {\footnotesize $x_{\inteval{\a-1}}$} (\inteval{\i-1}); 
}
\draw [draw = magenta, ->, thick] (7) edge node[pos=0.5, fill=white, inner sep = 1mm,sloped] {\footnotesize $x_{6}$}  (13);

\foreach \i [count=\a] in {15,...,20}{%
\draw [draw = MidnightBlue, ->, thick] (\i) edge node[pos=0.5, fill=white, inner sep=1mm,sloped] {\footnotesize $y_{\inteval{\a-1}}$} (\inteval{\i-1}); 
}
\draw [draw = MidnightBlue, ->, thick] (14) edge node[pos=0.5, fill=white, inner sep=1mm,sloped] {\footnotesize $y_{6}$}  (20);

\foreach \i [count=\a] in {21,...,25}{%
\draw [draw = ForestGreen, ->, thick] (\i) edge node[pos=0.5, fill=white, inner sep=1mm,sloped] {\footnotesize $z_{\inteval{\a+1}}$}(\inteval{\i+2}); 
}
\draw [draw = ForestGreen, ->, thick] (26) edge node[pos=0.5, fill=white, inner sep=1mm,sloped] {\footnotesize $z_{0}$} (\inteval{21}); 
\draw [draw = ForestGreen, ->, thick] (27) edge node[pos=0.5, fill=white, inner sep=1mm,sloped] {\footnotesize $z_1$} (\inteval{22}); 
\end{tikzpicture}
\]
Therefore, the skew group algebra $R\#G'$ can be presented as the path algebra of $Q'$ over $\C$ subject to the relations
\begin{align*} I' = \langle & x_i y_{i-1} - y_i x_{i-1}, x_i z_{i+2} - z_{i+3} x_{i+2}, y_i z_{i+2} - z_{i+3} y_{i+1} \rangle_{i \in \Z_7}.
\end{align*}
\end{example}

\subsection{New derived autoequivalences of skew group algebras} \label{new derived equivalences skew}

The goal of this section is to use the results of \cref{symmetric alg} in order to construct derived autoequivalences for skew-group algebras $R\#G$. To ease notation, let $S := R^G$ throughout, and consider the following set up.

\begin{setup} \label[setup]{quot setup}
    Under \cref{sl mckay setup}, suppose further that there is a ring surjection $p \colon R\#G \to B$, where $B$ is a finite dimensional $\C$-algebra which either is self-injective or has finite global dimension. Finally, suppose that $\ker p = (\ker p)^2$.
\end{setup}

\begin{notation}
    Let $\cR$ be a commutative ring. For each $\idlm \in \maxspec \cR$, let $\cR_{\idlm}$ denote the localisation of $\cR$ at $\idlm$. Moreover, given an $\cR$-module, let $M_{\idlm} = M \otimes_{\cR} \cR_{\idlm}$ denote its localisation at $\idlm$. 
\end{notation}

\subsubsection{Derived autoequivalence of $(R\#G)_\idlm$}

Since the results in \cref{symmetric alg} are stated for local rings, we consider maximal ideals $\idlm \in \maxspec R^G$ and first construct a spherical twist for the $S_\idlm$-algebra $(R\#G)_{\idlm}$. Thus, our first goal is to show that given $p \colon R\#G \to B$ as in \cref{quot setup}, then $p_{\idlm} \colon (R\#G)_{\idlm} \to B_{\idlm}$ satisfies the assumptions of \ref{3d spherical}.

\begin{remark} \label{gorenstein remark}
    Observe that since $G \subset \SL(3, \C)$, then $S$ is a Gorenstein ring \cite[6.4.9]{BH}. It follows that, for $\idlm \in \maxspec S$, then $S_{\idlm}$ is also Gorenstein.   
\end{remark}

\begin{lemma}[{\cite[5.16]{LW}}] \label{R is CM}
    Under \cref{quot setup}, let $\idlm \in \maxspec S$. Then, $R_{\idlm} \in \CM S_{\idlm}$
\end{lemma}
\begin{proof}
  Since $R_{\idlm} \in \modu S_{\idlm}$ and, moreover, $S_{\idlm}$ is Gorenstein by \ref{gorenstein remark}, it suffices to show that $\Ext^i_{S_{\idlm}}(R_{\idlm}, S_{\idlm}) = 0$ for all $i >0$. Well, from \cite[2.15]{DW1},
    \[ \Ext^i_{S_{\idlm}}(R_{\idlm}, S_{\idlm}) = \Ext^i_{S}(R, S) \otimes_{S} S_{\idlm} \]
    where the right-hand side vanishes because $R \in \CM S$  by \cite[5.16]{LW}. 
  
  %For convenience of the reader we summarise the proof  of 5.16 LW here. 
  
  %First, $R$ is a finitely generated $R^G$-module by \cite[5.4]{LW}. Further, there is an injective ring homomorphism $\rho \colon R \to R\#G$ defined by
  %\[ \rho(r) = \frac{1}{|G|} \sum_{g \in G} g(r) \otimes g \]
  %whose image is $\Img \rho = (R\#G)^G$. Thus, $R \cong (R\#G)^G$ as rings. It is straightforward to check that this implies $R \cong (R\#G)^G$ as $R^G$-modules. 

  %Observe that, by \cite[5.6]{LW}, $R\#G$ is a projective $R$-module. Whence, since the fixed point functor $(-)^G$ is exact, $(R\#G)^G$ is a projective $R^G$-module, and so the isomorphism $R \cong (R\#G)^G$ implies that $R$ is a projective $R^G$-module. Hence, since $R$ is finitely generated and projective, we may conclude $R \in \CM R^G$.
\end{proof}

\begin{lemma} \label{self inj local}
    Under \cref{quot setup}, let $\mathfrak{m} \in \maxspec S$. Then, $B_{\idlm}$ is either self-injective or has finite global dimension. 
\end{lemma}
\begin{proof}
    Note that since $B$ is finite dimensional, so is $B_{\idlm}$ (see e.g. \cite[2.13]{Ei}). Therefore, $B_{\idlm}$ is self-injective, if $\Ext^1_{B_{\idlm}}(M, B_{\idlm})$ vanishes for all $M \in \modu  B_{\idlm}$. On the other hand, $B_{\idlm}$ has finite global dimension, if there is an $n \in \N$ such that $\Ext^n_{B_{\idlm}}(M, N)$ vanishes for all $M, N \in \modu B_{\idlm}$.

    Next, observe that localisation induces an essentially surjective functor $\modu B \to \modu B_{\idlm}$ so that for any $M \in \modu  B_{\idlm}$, there exists an $\tilde{M} \in \modu B$ with $M = \tilde{M}_{\idlm}$. 

    Suppose first that $B$ is self-injective. Then, by \cite[2.15]{DW1},
    \[ \Ext^i_{B_{\idlm}}(M, B_{\idlm}) = \Ext^i_{B_{\idlm}}(\tilde{M}_{\idlm}, B_{\idlm}) = \Ext^i_B(\tilde{M}, B) \otimes R_{\idlm}. \]
    The right-hand side vanishes because $B$ is self-injective, and we may conclude that $B_{\idlm}$ is self-injective. 
    
    If $B$ is not self-injective then it must have finite global dimension. Again,  by \cite[2.15]{DW1}, 
    \[ \Ext^i_{B_{\idlm}}(M, N) = \Ext^i_{B_{\idlm}}(\tilde{M}_{\idlm}, \tilde{N}_{\idlm}) = \Ext^i_B(\tilde{M}, \tilde{N}) \otimes R_{\idlm}. \]
    Since $B$ has finite global dimension, there is an $n$ such that $\Ext^i_B(\tilde{M}, \tilde{N}) \otimes R_{\idlm} = 0$. Whence, $B_{\idlm}$ has finite global dimension.
\end{proof}

\begin{prop} \label{RG satisfies local setup}
     Assuming \cref{quot setup}, let $\mathfrak{m} \in \maxspec S$ and consider the surjective morphism $p_{\idlm} \colon (R\#G)_{\idlm} \to B_{\idlm}$ induced by localisation at $\idlm$. Then, 
    \begin{enumerate}
        \item $(R\#G)_{\idlm}$ has global dimension equal to $3$, \label{finite gdim local}
        \item $(R\#G)_{\idlm}$ satisfies the assumptions of \cref{module fin}, \label{setup satisfied local}
        \item $p_{\idlm} \colon (R\#G)_{\idlm} \to B_{\idlm}$ satisfies the assumptions of \cref{sym inj set up}, \label{gorder local}
        \item $\ker p_{\idlm} = (\ker p_{\idlm})^2$. \label{ideal cond local}
    \end{enumerate}
\end{prop}
\begin{proof}
    \eqref{finite gdim local} It is a classical fact (see e.g.\ \cite[5.6]{LW}) that $R\#G$ has global dimension equal to $3$. Since localisation is exact and essentially surjective, necessarily $(R\#G)_{\idlm}$ has global dimension equal to $3$.

    \eqref{setup satisfied local} It suffices to show that $(R\#G)_{\idlm}$ is a module finite $S_{\idlm}$-algebra. It is clear that the $R$-module $R\#G = R \otimes_\C \C G$ is a finite free $R$-module. Moreover, by \cite[5.4]{LW}, $R$ is a finite $S$-module, and so it follows that $R\#G$ is a finite $S$-module. Since localisation is exact, it must be that  $(R\#G)_{\idlm}$ is a module finite $S_{\idlm}$-algebra, as required.

    \eqref{gorder local} It follows from \ref{self inj local} that $B_{\idlm}$ is either self-injective or has finite global dimension. Moreover, \eqref{finite gdim local} implies that $B_{\idlm} \in \Kb(\proj (R\#G)_{\idlm})$. Thus, it suffices to prove that  $(R\#G)_{\idlm}$ is a Gorenstein $S_{\idlm}$-order. 
    
    By \ref{aus t}\eqref{aus 1}, $R\#G \cong \End_{S}(R)$. It thus follows from \cite[2.15]{DW1} that there is an isomorphism $(R\#G)_{\idlm} \cong \End_{S_{\idlm}}(R_{\idlm})$. Given this fact, and by the results \cite[2.22(2)]{IW1} and \cite[3.8 (1)$\Rightarrow$(3)]{IR}, then $(R\#G)_{\idlm} \cong \End_{S_{\idlm}}(R_{\idlm})$ is a Gorenstein $S_{\idlm}$-order if $R_{\idlm}$ is a reflexive $S_{\idlm}$-module and $\End_{S_{\idlm}}(R_{\idlm}) \in \CM S_{\idlm}$.

    Indeed, by \ref{R is CM}, $R_{\idlm} \in \CM S_{\idlm}$, and CM modules are reflexive. Thus, $R_{\idlm}$ is a reflexive $S_{\idlm}$-module. Moreover, by the arguments in the proof of \eqref{setup satisfied local}, $(R\#G)_{\idlm} \cong \End_{S_{\idlm}}(R_{\idlm})$ is a finite direct sum of copies of $R_{\idlm}$. Therefore, since $R_{\idlm} \in \CM S_{\idlm}$, then $\End_{S_{\idlm}}(R_{\idlm}) \in \CM S_{\idlm}$. Thus \cite[2.22(2)]{IW1} and \cite[3.8 (1)$\Rightarrow$(3)]{IR} imply \eqref{gorder local}.

    \eqref{ideal cond local} Since localisation is exact, $\ker p_{\idlm} = (\ker p)_{\idlm}$. For brevity, let $\ker p = I$. By assumption, $I = I^2$, so that that $I_{\idlm} = (I^2)_{\idlm}$. It follows straightforwardly from the definition of localisation of modules that $(I^2)_{\idlm} = (I_{\idlm})^2$, as required. 
\end{proof}

 Under \cref{quot setup}, with \ref{RG satisfies local setup} in hand, we are able to specify the twist and cotwist in the local setting. 

 \begin{notation} \label{beta not}
     Let $\idlm \in \maxspec R$. Observe that, as in \eqref{beta remark}, we may consider the morphisms $s \cdot \beta \colon  B \Dtensor_A B \to B$ in $\Dcat(\Mod B \otimes_\Z B^\op)$ and $s' \cdot \beta' \colon  B_{\idlm} \Dtensor_{A_{\idlm}} B_{\idlm} \to B_{\idlm}$ in $\Dcat(\Mod B_{\idlm} \otimes_\Z B_{\idlm}^\op)$.
 \end{notation}

\begin{cor} \label{local tilting RG}
    Under \cref{quot setup}, let $\mathfrak{m} \in \maxspec S$. Then,  the functor
    \[ F_{\idlm} \colon \Db(\modu B_\idlm) \to \Db(\modu (S\#G)_\idlm) \]
    is spherical. The twist and cotwist around $F_{\idlm}$ are
    \begin{align}
        & \cT_{\idlm} \colonequals \Rderived{\Hom}_{(R\#G)_{\idlm}}((\ker p)_{\idlm}, -) \colon \Db(\modu (R\#G)_{\idlm}) \to \Db(\modu (R\#G)_{\idlm}), \\
        & \cC_{\idlm} \colonequals  \Rderived{\Hom_B}( C(s' \cdot \beta'), -) \colon \Db(\modu B_{\idlm}) \to \Db(\modu B_{\idlm}),
    \end{align}
    respectively.
\end{cor}
\begin{proof}
    It follows from \ref{RG satisfies local setup} that the surjection $p_{\idlm} \colon (R\#G)_{\idlm} \to B_{\idlm}$ satisfies the assumptions of \ref{3d spherical}. Hence, the statement follows by noting that $(\ker p)_\idlm = \ker p_\idlm$ since localisation is exact. 
\end{proof}

Since the cotwist around $F_{\idlm}$ is controlled by $C(s' \cdot \beta')$, our goal in the remainder of this subsection is to prove that $C(s' \cdot \beta') \cong C(s \cdot \beta)_{\idlm}$.

\begin{lemma} \label{localisation commutes with tensor}
    Let $\cR$ be a commutative noetherian ring, and suppose that $\Lambda$ and $\Gamma$ are  module finite $\cR$-algebras. Let $M \in \modu \Gamma \otimes_{\cR} \Lambda^\op $ and $N \in \modu \Lambda \otimes_{\cR} \Gamma^\op$. Then, for each $\idlm \in \maxspec \cR$, $(M \Dtensor_{\Gamma} N)_{\idlm} \cong M_{\idlm} \Dtensor_{\Gamma_{\idlm}} N_{\idlm}$ in $\Dcat(\Mod \Lambda_{\idlm} \otimes_{\cR} \Lambda_{\idlm}^\op)$. 
\end{lemma}
\begin{proof}
    This proof is similar to \cite[3.2.10]{wei}. Consider a projective resolution $q \colon Q \to M$ in $\Dcat(\Mod \Gamma_{\idlm} \otimes_{\cR} \Gamma_{\idlm}^\op)$, where $Q$ is a chain complex which at degree $i$ has differential $d^{(i)} \colon Q^{(i)} \to Q^{(i-1)}$. Then, $(M \Dtensor_{\Gamma} N)_{\idlm}$ is quasi-isomorphic to the complex $L$ given by
    \[  L \colon \hdots \to (Q^{(i)} \otimes_{\Gamma} N)_{\idlm} \tow{(d^{(i)} \otimes_{\Gamma} N)_{\idlm}} (Q^{(i-1)} \otimes_{\Gamma} N)_{\idlm} \tow{(d^{(i-1)}\otimes_{\Gamma} N)_{\idlm}} \to \hdots \to (Q^{(0)} \otimes_{\Gamma} N)_{\idlm} \to 0 \to \hdots \]
    Using the definition of localisation, it is easy to check that, for each $i$, there is an $\Lambda_{\idlm}$-bimodule isomorphism 
    \[ \phi^{(i)} \colon (Q^{(i)} \otimes_{\Gamma} N)_{\idlm} = Q^{(i)} \otimes_{\Gamma} N \otimes_{\cR} \cR_{\idlm} \tow{\sim} Q^{(i)}_{\idlm} \otimes_{\Gamma_{\idlm}} N_{\idlm} \]
    sending $x \otimes_{\Gamma} m \otimes_{\cR} \frac{1}{s}$ to $\frac{x}{s} \otimes_{\Gamma_{\idlm}} \frac{m}{1}$. Moreover, $(d^{(i)} \otimes_{\Gamma} N)_{\idlm} \cdot \phi^{(i+1)} = \phi^{(i)} \cdot d^{(i+1)}_{\idlm} \otimes_{\Gamma_{\idlm}} N_{\idlm}$. Thus, $L$ is isomorphic to the complex $L'$
     \[  L' \colon \hdots \to Q^{(i)}_{\idlm} \otimes_{\Gamma_{\idlm}} N_{\idlm} \tow{d^{(i)}_{\idlm} \otimes_{\Gamma_{\idlm}} N_{\idlm}} Q^{(i-1)}_{\idlm} \otimes_{\Gamma_{\idlm}} N_{\idlm} \tow{d^{(i-1)}_{\idlm} \otimes_{\Gamma_{\idlm}} N_{\idlm}} \to \hdots \to Q^{(0)}_{\idlm} \otimes_{\Gamma_{\idlm}} N_{\idlm} \to 0 \to \hdots \]

    Next, observe that since localisation is exact, the complex $Q_{\idlm}$ given by
    \[ \hdots \to Q^{(i)}_{\idlm} \tow{d^{(i)}_{\idlm}} Q^{(i-1)}_{\idlm} \tow{d^{(i-1)}_{\idlm}} \to \hdots \]
    induces a projective resolution $q_{\idlm} \colon Q_{\idlm} \to M_{\idlm}$. Thus,  $M_{\idlm} \Dtensor_{\Gamma_{\idlm}} N_{\idlm}$ is quasi-isomorphic to $L'$, and, therefore, the statement follows.
\end{proof}

\begin{lemma} \label{C of local is local of C}
    With notation as in \ref{beta not}, there is a commutative diagram 
    \[ \begin{tikzcd}
        (B \Dtensor_A B)_{\idlm} \arrow[d, "\sim"{anchor=north, rotate=90}, "\phi"'] \arrow[r, "{(s \cdot \beta)_{\idlm}}"] & B_{\idlm} \arrow[d, equals] \\
        B_{\idlm} \Dtensor_{A_{\idlm}} B_{\idlm} \arrow[r, "s' \cdot \beta'"] & B_{\idlm}.
    \end{tikzcd}
    \]
    Thus, $ C(s' \cdot \beta') \cong  C(s \cdot \beta)_{\idlm}$.
\end{lemma}
\begin{proof}
    We let $\phi \colon (B \Dtensor_A B)_{\idlm} \to B_{\idlm} \Dtensor_{A_{\idlm}} B_{\idlm}$ be the isomorphism of \ref{localisation commutes with tensor}. It is then straightforward to check by diagram chasing that the diagram in the statement commutes.

    Therefore, 
    \[ C(s' \cdot \beta') = \cone(s' \cdot \beta') \cong \cone(s \cdot \beta)_{\idlm}. \]
    Moreover,  the $B$-bimodule $C(s \cdot \beta)$ is defined by the triangle
    \[  B \Dtensor_A B \tow{s \cdot \beta} B \to C(s \cdot \beta) \to^+ \]
    which, by exactness of localisation, induces the triangle of $B_{\idlm}$-bimodules
    \[ (B \Dtensor_A B )_{\idlm} \tow{(s \cdot \beta)_{\idlm}} B_{\idlm} \to C(s \cdot \beta)_{\idlm} \to^+. \]
    Hence, $\cone(s \cdot \beta)_{\idlm} \cong  C(s \cdot \beta)_{\idlm}$, and the statement follows.
\end{proof}

\subsubsection{Derived autoequivalence of $R\#G$}

Under \cref{quot setup}, $R\#G$ has finite global dimension (Here, since $R\#G$ is noetherian, there is no need distinguish between right and left global dimension). Hence, the algebra $B$ is perfect as a right and left $R\#G$-module. Thus, the restriction of scalars functor $F = - \otimes_B B \colon \Dcat(\Mod B) \to \Dcat(\Mod R\#G)$, as well as its right and left adjoints, preserve bounded complexes of finitely generated modules. Thus, the twist and cotwist around $F \colon \Db(\modu B) \to \Db(\modu R\#G)$ can be readily specified as a corollary of \ref{twist for epi} and \ref{cotwist gen}.

\begin{cor} \label{tw ctw RG}
    Under \cref{quot setup}, the twist and cotwist around $F \colon \Db(\modu B) \to \Db(\modu R\#G)$ are
\begin{align*}
    & \cT = \Rderived{\Hom_{R\#G}}(\ker p, -) \colon \Db(\modu R\#G) \to \Db(\modu R\#G), \\
    & \cC = \Rderived{\Hom_{B}}(C(s \cdot \beta), -) \colon \Db(\modu B) \to \Db(\modu B).
\end{align*}
\end{cor}
\begin{proof}
    This is an immediate consequence of \ref{twist for epi}, \ref{cotwist gen}, and the fact that $F$ (as well as its right and left adjoints) preserve bounded complexes of finitely generated modules. 
\end{proof}

In what follows, we will show that $\cT$ and $\cC$ must be equivalences. For this, we will need two key propositions.

\begin{prop} \label{local tilting implies global}
    Consider a field $\mathbb{K}$, and let $\cR$ be a commutative noetherian $\mathbb{K}$-algebra. Suppose that $\Gamma$ is a module finite $\cR$-algebra, and let $M$ be a biperfect complex of finitely generated $\Gamma$-bimodules. If $M_\idlm$ is tilting for all $\idlm \in \maxspec \cR$, then $M$ is tilting.
\end{prop}
\begin{proof}
   We will show that $M$ is tilting because it satisfies (a)-(c) of \cite[1.8]{mi}. Namely, 
    \begin{enumerate}[label=(\alph*)]
        \item The bimodule complex $M$ is biperfect. 
        \item The left multiplication map $\lambda \colon \Gamma \to \Rderived{\Hom_{\Gamma}}(M, M)$ is an isomorphism in $\Dcat(\Mod \Gamma \otimes_k \Gamma^\op)$
        \item Right multiplication $\rho \colon \Gamma \to \Rderived{\Hom_{\Gamma^\op}}(M, M)$is an isomorphism in $\Dcat(\Mod \Gamma \otimes_k \Gamma^\op)$.
    \end{enumerate}

   Observe that (a) is true by assumption. To prove (b), we aim to show that the natural left multiplication map $\lambda \colon \Gamma \to \Rderived{\Hom_\Gamma}(M, M)$ is a bimodule isomorphism in the $\Dcat(\Mod \Gamma \otimes_k \Gamma^\op)$. First, since $M_\idlm$ is tilting for all ideals $\idlm \in \maxspec R$, then the natural left multiplication map $\Gamma_\idlm \to \Rderived{\Hom_{\Gamma_\idlm}}(M_\idlm, M_\idlm)$ is a quasi-isomorphism. That is, $\Rderived{\Hom_{\Gamma_\idlm}}(M_\idlm, M_\idlm)$ is concentrated in degree zero. By  \cite[2.15]{DW1}, and because localisation is exact, for $i \neq 0$,
   \[0 = \cohom{i} \left( \Rderived{\Hom_{\Gamma_\idlm}}(M_\idlm, M_\idlm) \right) \cong \cohom{i} \left( \Rderived{\Hom_{\Gamma}}(M, M) \right) \otimes_{\cR} \cR_{\idlm}. \]
   Since the module $\cohom{i} \left( \Rderived{\Hom_{\Gamma}}(M, M) \right) \otimes_\cR \cR_\idlm$ vanishes for all maximal ideals $\idlm \in \maxspec R$, necessarily $\cohom{i} \left( \Rderived{\Hom_{\Gamma}}(M, M) \right) = 0$ for $i \neq 0$. Therefore, it suffices to show that 
   \[ \cohom{0}(\lambda) \colon \Gamma \to \cohom{0} \left( \Rderived{\Hom_{\Gamma}}(M, M) \right) \cong \Hom_{\Dcat(\Gamma)}(M, M) \]
   is an isomorphism. 
   
   Well, since $M_\idlm$ is tilting, the natural multiplication map 
   \[ \gamma_\idlm \colon \Gamma_\idlm \to  \cohom{0} \left( \Rderived{\Hom_{\Gamma_\idlm}}(M_\idlm, M_\idlm) \right) \cong \Hom_{\Dcat(\Gamma_\idlm)}(M_\idlm, M_\idlm)\]
   is an isomorphism. Moreover, it straightforward to check that, under the identification, 
   \[ \Hom_{\Dcat(\Gamma)}(M, M) \otimes \cR_{\idlm} \cong \cohom{0} \left( \Rderived{\Hom_{\Gamma_\idlm}}(M_\idlm, M_\idlm) \right) \cong \Hom_{\Dcat(\Gamma_\idlm)}(M_\idlm, M_\idlm)\]
   then $\cohom{0}(\lambda) \otimes_\cR \cR_\idlm = \gamma_\idlm$. Hence, $\cohom{0}(\lambda) \otimes_\cR \cR_\idlm$ is an isomorphism for all $\idlm \in \maxspec \cR$. Thus, $\cohom{0}(\lambda)$ must be an isomorphism, as required. 
   
   The dual argument works to show that (c) holds. That is, that the natural left multiplication map is an isomorphism. 
\end{proof}

\begin{prop} \label{c local implies z local pre RG}
    Consider a field $\mathbb{K}$, and let $\cR$ be a commutative noetherian $\mathbb{K}$-algebra. Suppose that $\Gamma$ is a module finite $\cR$-algebra, and let $M$ be a finitely generated $\Gamma$-bimodule.  If there is an $n$ such that for all $\idlm \in \maxspec \cR$, $\pdim_{\Gamma_\idlm}(M_\idlm) \leq n$, $\pdim_{\Gamma_\idlm^\op}(M_\idlm) \leq n$, and $M_\idlm$ is tilting, then $M$ is tilting.
\end{prop}
\begin{proof}
    In view of \ref{local tilting implies global}, it suffices to show that $M$ is perfect as a right and a left $\Gamma$-module. Indeed, consider a finitely generated $\Gamma$-module $N$. Then, for any $\idlm \in \maxspec \cR$, it follows from \cite[2.15]{DW1} that
   \[ \Ext^i_{\Gamma}(M, N) \otimes_{\cR} \cR_{\idlm} = \Ext^i_{\Gamma_\idlm}(M_\idlm, N_\idlm). \]
   Therefore, since $\pdim_{\Gamma_\idlm}(M_\idlm) \leq n$, necessarily  $\Ext^i_{\Gamma}(M, N) \otimes_{\cR} \cR_{\idlm} = 0$ for all $i > n$ and maximal ideals $\idlm \in \maxspec \cR$. Whence, $\Ext^i_{\Gamma}(M, N) = 0$ for all $i > n$, so that $\pdim_\Gamma M \leq n$, as required. A similar same argument works to show that $\pdim_{\Gamma^\op} M \leq n$.
\end{proof}

In light of these propositions, we are able to prove that the twist and cotwist in \ref{tw ctw RG} are equivalences, and so the functor $F$ is spherical 

\begin{theorem} \label{spherical RG}
    Under \cref{quot setup}, the functor
    \[ F = - \Dtensor_B B \colon \Db(\modu B) \to \Db(\modu S\#G) \]
    is spherical. 
\end{theorem}
\begin{proof}
    Due to \ref{tw ctw RG}, the twist and cotwist around $F \colon \Db(\modu B) \to \Db(\modu R\#G)$ are
\begin{align*}
    & \cT = \Rderived{\Hom_{R\#G}}(\ker p, -) \colon \Db(\modu R\#G) \to \Db(\modu R\#G), \\
    & \cC = \Rderived{\Hom_{B}}(C(s \cdot \beta), -) \colon \Db(\modu B) \to \Db(\modu B).
\end{align*}
    We will show that $F$ is spherical by proving that both $\cT$ and $\cC$ must be equivalences. 
    
    To see that $\cT$ is an equivalence, note that \ref{local tilting RG} implies that $(\ker p)_\idlm$ is tilting for all ideals $\mathfrak{m} \in \maxspec S$. Moreover, \ref{RG satisfies local setup} \eqref{finite gdim local} and its dual implies that $\pdim_{(R\#G)_\idlm}((\ker p)_\idlm) \leq 3$ and $\pdim_{(R\#G)_\idlm^\op}((\ker p)_\idlm) \leq 3$. Hence, $\ker p$ is tilting by \ref{c local implies z local pre RG}, as required.

    Similarly, to see that $\cC$ is an equivalence, note that \ref{local tilting RG} and \ref{C of local is local of C} imply that $C(s \cdot \beta)_\idlm$ is a bimodule tilting complex for all $\mathfrak{m} \in \maxspec S$. Thus, in light of \ref{local tilting implies global}, it suffices to show that $C(s \cdot \beta)$ is biperfect.
    
    Suppose first that $B$ has finite global dimension. Then, since $C(s \cdot \beta)$ is bounded as a complex of right and a left $B$-modules by \ref{C is bounded}, necessarily it is biperfect. We may thus conclude from \ref{local tilting implies global} that $C(s \cdot \beta)$ is a bimodule tilting complex. 
    
    If, on the other hand, $B$ is self-injective, then we claim that $C(s \cdot \beta)$ seen as a right or left $B$-module is a shifted projective module and, thus, it is biperfect. 
    
    We first show that there is an isomorphism $C(s \cdot \beta) \cong \Tor^{R\#G}_{3}(B, B)[-4]$. Due to \ref{ctw res cohom in 2 degrees}, it suffices to show that $\cohom{k}(B \Dtensor_{R\#G} DB) = 0$ for all $k \neq 0, -3$. Well, by \ref{localisation commutes with tensor} and the fact that localisation is exact,
    \[ \cohom{k}(B \Dtensor_{R\#G} DB) \otimes_{\cR} \cR_{\idlm} \cong \cohom{k}(B_{\idlm} \Dtensor_{{R\#G}_{\idlm}} DB_{\idlm}) \]
    for all $\idlm \in \maxspec \cR$. Moreover, by assumption, $\ker p_{\idlm} = (\ker p_{\idlm})^2$ and so \ref{CMness} and \ref{tor vanishing} imply that $\cohom{k}(B_{\idlm} \Dtensor_{{R\#G}_{\idlm}} DB_{\idlm})$ vanishes for all $k \neq 0, -3$. Hence, $\cohom{k}(B \Dtensor_{R\#G} DB) \otimes_{\cR} \cR_{\idlm}$ vanishes on all  $\idlm \in \maxspec \cR$, whence $\cohom{k}(B \Dtensor_{R\#G} DB) = 0$, as required. 

    Since $C(s \cdot \beta) \cong \Tor^{R\#G}_{3}(B, B)[-4]$, consider the following chain of $B_{\idlm}$-bimodule isomorphisms
     \begin{align*}
         \Tor^{R\#G}_{3}(B, B) \otimes_{\cR} \cR_{\idlm} & \cong \Tor^{(R\#G)_{\idlm}}_{3}(B_{\idlm}, B_{\idlm}) \tag{By \ref{localisation commutes with tensor}} \\
          & \cong \Hom_{(R\#G)_{\idlm}}(DB_{\idlm}, B_{\idlm}) \tag{By \ref{tor ext duality}} \\
          & \cong \Hom_{B_{\idlm}}(DB_{\idlm}, B_{\idlm}) 
     \end{align*}
     Now, since $B_{\idlm}$ is self-injective, then $DB_{\idlm}$ is a projective left and right $B_{\idlm}$-module. Whence, $\Hom_{B_{\idlm}}(DB_{\idlm}, B_{\idlm})$ is projective on either side. Thus, the proof of \ref{c local implies z local pre RG} implies that $C(s \cdot \beta)$ is projective on either side and, thus, is biperfect.  Consequently, \ref{local tilting implies global} implies that $C(s \cdot \beta)$ is a bimodule tilting complex. 
    \end{proof}

\subsection{New derived autoequivalences of $G$-$\Hilb(\C^3)$} \label{sec: GHilb}

In order to construct a spherical twists for the variety $G$-$\Hilb(\C^3)$, the following lemma will be important. 

\begin{lemma} \label{spherical composed with equiv}
    Let $S \colon \cA \to \cB$ be a functor between triangulated categories which admits right adjoint $R$ and left adjoint $L$, and let $E \colon \cB \to \cB'$ be an equivalence of triangulated categories. If $T \colon \cB \to \cB$ is the twist around $S$, then $E \cdot T \cdot E^{-1} \colon \cB' \to \cB'$ is the twist around $E \cdot S \colon \cA \to \cB'$. Moreover, if $C \colon \cA \to \cA$ is the cotwist around $S$, then it is also the cotwist around $E \cdot S$. Hence, if $S$ is spherical, so is $E \cdot S$.
\end{lemma}
\begin{proof}
    In this proof we will use the less technical adapted definition of twist and cotwist as in \ref{twist def}. For a more rigorous treatment, see e.g.\ \cite[6.8]{Go}.
    
    For each $a \in \cA$ and $b \in \cB$, write the Hom-set isomorphism induced by the adjoint pair $(S, R)$ as
    \[ \Phi_{(a, b)} \colon \Hom_{\cA}(a, R(b)) \to \Hom_\cB(S(a), b). \]
    Let $\eta$ and $\epsilon$ be the unit and counit of this adjunction, respectively.
    
    It is easy to check that the composition $E \cdot S$ admits right adjoint $R \cdot E^{-1}$ since for any $a \in \cA$ and $c \in \cB'$, there are natural isomorphisms
    \[ \Psi_{a, c} \colon \Hom_\cA(a, R \cdot E^{-1}(c) ) \tow{E \Phi_{(a, E^{-1}(c))} } \Hom_{\cB'}(E \cdot S(a), c). \] 
   Hence, the unit $\eta^{R \cdot E^{-1}}$ of the adjunction $(E \cdot S, R \cdot E^{-1})$ is specified by
    \[ \eta^{R \cdot E^{-1}}_{a} = \Psi^{-1}_{(a, E \cdot S(a))}(1_{E \cdot S(a)}) = \Phi^{-1}_{(a, E^{-1} \cdot E \cdot Sa)} \cdot E^{-1}(1_{E \cdot S(a)}) = \Phi_{(a, Sa)}(1_{Sa}) = \eta_{a} \]
    for all $a \in \cA$. Similarly, the counit $\epsilon^{R \cdot E^{-1}}$ is specified by 
    \[ \epsilon^{R \cdot E^{-1}}_{c} = \Psi_{(R \cdot E^{-1}(c), c)}(1_{R \cdot E^{-1}(c)}) = E \Phi_{(R \cdot E^{-1}(c), E^{-1}(c))}(1_{R \cdot E^{-1}(c)}) = E \epsilon_{ E^{-1}(c)} \]
    for all $c \in \cB'$. 
    
    Now, since $T$ is the twist around $S$, there is a functorial triangle
    \[ S \cdot R \tow{\epsilon} 1_\cB \to T \to^+ \]
    in $\cB$ which induces the triangle
    \[ E \cdot S \cdot R \cdot (E)^{-1} \tow{ E \epsilon_{E^{-1}} } 1_{\cB'} \to E \cdot T \cdot E^{-1} \to^+ \]
    in $\cB'$. By the previous argument, $E \epsilon_{E^{-1}}$ is the counit of the adjoint pair $(E \cdot S, R \cdot E^{-1})$. 

    The fact that the cotwist is the same follows from the equality $ \eta^{R \cdot E^{-1}} =  \eta$.
    \end{proof}

\begin{theorem} \label{quot them}
    Assume \cref{quot setup}. Then, the functor
    \[ F_G = \Phi_G \cdot - \Dtensor_B B \colon \Db(\modu B) \to \Db(\coh G \text{-} \Hilb(\C^3)) \]
    is spherical. The twist and cotwist around $F_G$ are
    \begin{align*}
        & \cT_G =  \Phi_G \cdot \Rderived{\Hom_{R\#G}(\ker p, -)} \cdot \Phi_G^{-1} \colon \Db(\coh G \text{-} \Hilb(\C^3)) \to \Db(\coh G \text{-} \Hilb(\C^3)), \\
        & \cC_G = \Rderived{\Hom_B}(C(s \cdot \beta), -) \colon \Db(\modu B) \to \Db(\modu B),
    \end{align*}
    respectively. 
\end{theorem}
\begin{proof}
    This follows by combining \ref{spherical RG} and \ref{spherical composed with equiv}.
\end{proof}

\begin{example} \label{example ghilb}
    As a first example, consider the group $G \subset \SL(3, \C)$ 
    \[ G = \left \langle \begin{pmatrix} \epsilon_3 & 0 & 0 \\ 0 & \epsilon_3 & 0 \\ 0 & 0 & \epsilon_3 \end{pmatrix}\right \rangle \]
    as in \ref{mckay example}. Then, the skew group algebra $R\#G$ can be presented as the path algebra of the quiver $Q$ subject to some relations $I$ as in \ref{mckay example}.

    Let $e_i$ denote the idempotent of $R \#G$ corresponding to vertex $i$ in $Q$, and consider the two-sided ideal $\langle e_i \rangle$ of $R\#G$. Then, the quotients $B_i := R\#G / \langle e_i \rangle$ can be presented as the path algebra of the 3-Kronecker quiver and thus have finite global dimension. 

Therefore, it follows from \ref{quot them} that the functor
    \[\Phi_G \cdot \Rderived{\Hom_{R\#G}(\langle e_i \rangle, -)} \cdot \Phi_G^{-1} \colon \Db(\coh G \text{-} \Hilb(\C^3)) \to \Db( \coh G \text{-} \Hilb(\C^3)) \]
    is the spherical twist around $F_G = \Phi_G \cdot - \Dtensor_B B \colon \Db(\modu B) \to \Db(\coh G \text{-} \Hilb(\C^3))$.
\end{example}

\subsection{Cyclic groups of odd order.} \label{sec: example}

More generally, let $n>3$ be an odd number, $\epsilon_n$ a primitive $n$th root of unity and consider the group
\[ G = \left \langle \begin{pmatrix} \epsilon_n & 0 & 0 \\ 0 & \epsilon_n & 0 \\ 0 & 0 & \epsilon_n^{n -2} \end{pmatrix}\right \rangle \]
acting on $R = \C[x, y, z]$ as in previous examples. Then, $R\#G$ can be presented as the path algebra of a quiver $Q$ with vertices $Q_0 = \{0, 1, \ldots n-1\}$, arrows
    \begin{align*}
        &  k+1 \tow{x_{k} } k, \\
        &  k+1 \tow{y_{k} } k, \\
        &  k-2 \tow{z_{k} } k. 
    \end{align*}
    for all $0 \leq k \leq n-1$ and relations
\begin{align*} I = \langle & x_i y_{i-1} - y_i x_{i-1}, x_i z_{i+2} - z_{i+3} x_{i+2}, y_i z_{i+2} - z_{i+3} y_{i+1} \rangle_{i \in \Z_n}.
\end{align*}
Note that these relations intuitively mean that the arrows corresponding to $x, y, z$ "commute". The case $n = 7$ is \ref{mckay example 2}. 

We will specify precisely which quotients of $R\#G$ are finite dimensional algebras with finite global dimension. To do so, we will need a handful of lemmas.

\begin{lemma} \label{cycle types}
    A cycle at vertex $i$ in $\C Q/I$ must factor through at least one of the following families 
    \begin{align}
        & \alpha_{i-1} \ldots \alpha_{i-n}, \label{f1} \\
        & \alpha_{i-1} \alpha_{i-2} z_{i}, \label{f2} \\
        & \alpha_{i-1} z_{i + 1} z_{i+3} \ldots z_{i+n}, \label{f3} \\
        & z_{i+2} z_{i+4} \ldots z_{i + 2n} \label{f4}
    \end{align}
    for all $i \in \Z_n$ and $\alpha_i \in \{x_i, y_i\}$.
\end{lemma}
\begin{proof}
    Suppose $p$ is a cycle at vertex $i$ in $Q$. If $p$ does not factor through $z_j$ for any $j$, then $p$ is a composition of $x_j$'s and $y_j$'s. Hence, since $p$ must start and end at $i$, it must be that $p = (\alpha_{i-1} \ldots \alpha_{i-n})^r$ for some positive integer $r$ and $\alpha_i \in \{x_i, y_i\}$. Clearly, $p$ factors through \eqref{f1}. 

    On the other hand, if $p$ does not factor through $x_j$ or $y_j$ for any $j$, then $p$ is a composition of $z_j$'s and so necessarily $p = (z_{i+2} z_{i+4} \ldots z_{i + 2n})^r$, otherwise $p$ would not start and end at $i$. Hence, $p$ factors through \eqref{f4}.

   It remains to consider the case where $p$ factors through $z_s$ and through $\alpha_t \in \{x_t. y_t\}$ for some $s$ and $t$. Since the "variables" $x, y, z$ "commute" via the relations in $I$, then we may write the path $p = \alpha_{i-1} z_{i+1} p'$ where $p'$ is a path $p' \colon i+1 \to i$. 
   
   Now, if $p'$ does not factor through $x_j$ or $y_j$ then it is a composition of $z_j's$ and we may write $p' = (z_{i+3} z_{i+ 5} \ldots z_{i + n})(z_{i+2} z_{i+4} \ldots z_{i + 2n})^r$. Thus, $p$ factors through \eqref{f3}. 
    
    If, however, $p'$ does factor through $\alpha_m \in \{x_m, y_m\}$ for some $m$, then we may write $p' = \alpha_i p^{''}$. Thus, $p = \alpha_{i-1} z_{i+1} \alpha_i p^{''}$ which is equivalent to $\alpha_{i-1} \alpha_{i-2} z_{i} p^{''}$ and thus factors through \eqref{f2}.
\end{proof}

\begin{notation}
    Let $e \in R\#G$ be an idempotent not equal to $0$. Since the primitive idempotents of $R\#G$ are the vertex idempotents $e_i$ for $i \in Q_0$, then there is a set $V \subset Q_0$ such that  $e = \sum_{i \in V} e_i$.

   Let $Q_{\hat{V}}$ denote the subquiver of $Q$ obtained by deleting the vertices in $V$ and any arrows which enter or leave the deleted vertices. Given this notation, it is straightforward to show that $\C Q = \C Q_{\hat{V}} \oplus \C Q \, e \, \C Q$. Hence, for any path $\alpha \in \C Q$, we may write $\alpha = \alpha_{\hat{V}} + \alpha_V$. Therefore, let 
    \begin{align*}
       I_{\hat{V}} & = \langle \{ \alpha_{\hat{V}} \mid \alpha \in I \} \rangle.
   \end{align*} 
   It is straightforward to check that $R\#G/ \langle e \rangle \cong \C Q_{\hat{V}}/ I_{\hat{V}}$. 
\end{notation}

\begin{lemma} \label{quot is fd}
    Let $e \in R\#G$ be an idempotent not equal to $0$. Then, $R\#G/ \langle e \rangle$ is finite dimensional.
\end{lemma}
\begin{proof}
    We will show that for any choice of cycles $p_1, \ldots, p_n$ in $\C Q_{\hat{V}}/ I_{\hat{V}}$ then their composition is in $I_{\hat{V}}$. Thus, $R\#G/ \langle e \rangle$ is finite dimensional. 
    
    Let $p$ be a cycle in $\C Q_{\hat{V}}/I_{\hat{V}}$, then it is a cycle in $\C Q/I$. By \ref{cycle types}, $p$ must factor through at least one of the following families of cycles
    \begin{align*}
        & \alpha_{i-1} \ldots \alpha_{i-n},  \\
        & \alpha_{i-1} \alpha_{i-2} z_{i}, \\
        & \alpha_{i-1} z_{i + 1} z_{i+3} \ldots z_{i+n}, \\
        & z_{i+2} z_{i+4} \ldots z_{i + 2n}.
    \end{align*}
    However, note that $z_{i+2} z_{i+4} \ldots z_{i + 2n}$ and $\alpha_{i-1} \ldots \alpha_{i-n}$ factor through all vertices in $Q_0$ and hence factor through vertices in $V$. Therefore, $p$ cannot factor through these two families and must either factor through $\alpha_{i-1} \alpha_{i-2} z_{i}$ or $\alpha_{i-1} z_{i + 1} z_{i+3} \ldots z_{i+n}$. In either case, necessarily $p$ factors through at least one $\alpha_i$. 

    Therefore, since the variables $x, y, z$ "commute" via the relations in $I$, given any choice of cycles $p_1, \ldots, p_n$ then the composition 
    \begin{align*}
        p_1 \ldots p_n - \alpha_{i-1} \alpha_{i-2} \ldots \alpha_{i-n} p^{'} \in I.
    \end{align*}
    Now $p_1 \ldots p_n \in \C Q_{\hat{V}}$ by construction and so $(p_1 \ldots p_n)_{\hat{V}} =  p_1 \ldots p_n$. Moreover, $\alpha_{i-1} \ldots \alpha_{i-n}$ factors through vertices in $V$ and so $(\alpha_{i-1} \alpha_{i-2} \ldots \alpha_{i-n} p^{'})_{\hat{V}} = 0$. Thus,
        \begin{align*}
        p_1 \ldots p_n = (p_1 \ldots p_n - \alpha_{i-1} \alpha_{i-2} \ldots \alpha_{i-n} p^{'})_{\hat{V}} \in I_{\hat{V}}.
    \end{align*}
\end{proof}

\begin{prop} \label{fin gdim}
    Let $e \in R\#G$ be an idempotent not equal to $0$. If $Q_{\hat{V}}$ seen as a subquiver of $Q$ does not contain four successive vertices, then $R\#G/ \langle e \rangle$ has finite global dimension.
\end{prop}
\begin{proof} 
     Since $R\#G/ \langle e \rangle$ is finite dimensional by \ref{quot is fd}, it has finite global dimension if and only if all of its simple modules have finite projective dimension. Hence, we will show that every simple module of $R\#G/ \langle e \rangle$ admits a finite projective resolution.

    Since $Q_{\hat{V}}$ does not contain four successive vertices then, up to a relabelling, the neighbourhood around a vertex in $Q_{\hat{V}}$ falls into one of three cases
    
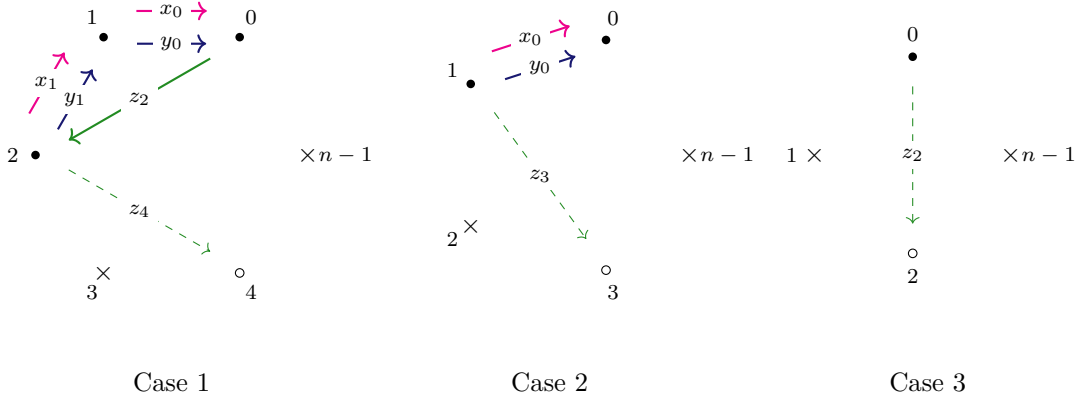
\[
    \begin{tikzpicture}
        \foreach \b [count=\a] in {0,...,5}{%
        \ifthenelse{\a < 4}{
            \draw (\a*360/6: 1.8cm) node [circle, fill=black, inner sep =0.4mm, align=center] (\a) {};
            \draw (\a*360/6: 2.1cm) node [circle, align=center]{ \footnotesize \b};
            }{
                \ifthenelse{\a = 5}{
                    \draw (\a*360/6: 1.8cm) node [circle, draw=black, inner sep =0.4mm, align=center] (\a) {};
                    \draw (\a*360/6: 2.1cm) node [circle, align=center]{\footnotesize \b};
                }{ 
                    \draw (\a*360/6: 1.8cm) node [circle, inner sep =0.4mm, align=center] (\a) {$\times$};
                    \ifthenelse{\a < 6}{
                    \draw (\a*360/6: 2.1cm) node [circle, align=center]{\footnotesize \b};}{
                    \draw (\a*360/6: 2.3cm) node [circle, text width=10mm, align=center] (\b) {\footnotesize $n-1$};}
                }
            }
        }
        \foreach \b [count=\a] in {6,...,11}{%
        \draw (\a*360/6: 2.2cm) node [circle, text width=10mm, align=center] (\b) {};
}
        \foreach \b [count=\a] in {12,...,17}{%
        \draw (\a*360/6: 1.7cm) node [circle, text width=5mm, align=center] (\b) {};
}

    \draw [draw = magenta, ->, thick] (8) edge node[pos=0.5, fill=white, inner sep=1mm] {\footnotesize $x_{1}$} (7); 
    \draw [draw = MidnightBlue, ->, thick] (14) edge node[pos=0.5, fill=white, inner sep=1mm] {\footnotesize $y_{1}$} (13); 

    \draw [draw = magenta, ->, thick] (7) edge node[pos=0.5, fill=white, inner sep=1mm] {\footnotesize $x_{0}$} (6); 
    \draw [draw = MidnightBlue, ->, thick] (13) edge node[pos=0.5, fill=white, inner sep=1mm] {\footnotesize $y_{0}$} (12); 

    \draw [draw=ForestGreen, ->, thick] (12) edge node[pos=0.5, fill=white, inner sep=1mm] {\footnotesize $z_{2}$} (14); 
     \draw [draw=ForestGreen, ->, dashed] (14) edge node[pos=0.5, fill=white, inner sep=1mm] {\footnotesize $z_{4}$} (16); 
    
    \draw (0, -3) node {Case 1};

%%%%%%%%%%%%%%
           \foreach \b [count=\a] in {0,...,4}{%
        \ifthenelse{\a < 3}{
            \draw ([xshift=5.25cm] \a*360/5: 1.6cm) node [circle, fill=black, inner sep =0.4mm, align=center] (\a) {};
            \draw ([xshift=5.25cm] \a*360/5: 1.9cm) node [circle, align=center]{\footnotesize \b};
            }{
                \ifthenelse{\a = 4}{
                    \draw ([xshift=5.25cm] \a*360/5: 1.6cm) node [circle, draw=black, inner sep =0.4mm, align=center] (\a) {};
                    \draw ([xshift=5.25cm] \a*360/5: 1.9cm) node [circle, align=center]{\footnotesize \b};
                }{ 
                    \draw ([xshift=5.25cm] \a*360/5: 1.6cm) node [circle, inner sep =0.4mm, align=center] (\a) {$\times$};
                    \ifthenelse{\a < 5}{
                    \draw ([xshift=5.25cm] \a*360/5: 1.9cm) node [circle, align=center]{\footnotesize \b};}{
                    \draw ([xshift=5.25cm] \a*360/5: 2.1cm) node [circle, text width=10mm, align=center] (\b) {\footnotesize $n-1$};}
                }
            }
        }
        \foreach \b [count=\a] in {5,...,9}{%
        \draw ([xshift=5.25cm] \a*360/5: 2cm) node [circle, text width=10mm, align=center] (\b) {};
}
        \foreach \b [count=\a] in {10,...,14}{%
        \draw ([xshift=5.25cm] \a*360/5: 1.5cm) node [circle, text width=5mm, align=center] (\b) {};
}

    \draw [draw = magenta, ->, thick] (6) edge node[pos=0.5, fill=white, inner sep=1mm] {\footnotesize $x_{0}$} (5); 
    \draw [draw = MidnightBlue, ->, thick] (11) edge node[pos=0.5, fill=white, inner sep=1mm] {\footnotesize $y_{0}$} (10); 
    \draw [draw=ForestGreen, ->, dashed] (11) edge node[pos=0.5, fill=white, inner sep=1mm] {\footnotesize $z_{3}$} (13); 
    
    \draw (5, -3) node {Case 2};

    %%%%%%%%%%%%%%
           \foreach \b [count=\a] in {0,...,3}{%
        \ifthenelse{\a = 1}{
            \draw ([xshift=9.8cm] \a*360/4: 1.3cm) node [circle, fill=black, inner sep =0.4mm, align=center] (\a) {};
            \draw ([xshift=9.8cm] \a*360/4: 1.6cm) node [circle, align=center]{\footnotesize \b};
            }{
                \ifthenelse{\a = 3}{
                    \draw ([xshift=9.8cm] \a*360/4: 1.3cm) node [circle, draw=black, inner sep =0.4mm, align=center] (\a) {};
                    \draw ([xshift=9.8cm] \a*360/4: 1.6cm) node [circle, align=center]{\footnotesize \b};
                }{ 
                    \draw ([xshift=9.8cm] \a*360/4: 1.3cm) node [circle, inner sep =0.4mm, align=center] (\a) {$\times$};
                    \ifthenelse{\a < 4}{
                    \draw ([xshift=9.8cm] \a*360/4: 1.6cm) node [circle, align=center]{\footnotesize \b};}{
                    \draw ([xshift=9.8cm] \a*360/4: 1.8cm) node [circle, text width=10mm, align=center] (\b) {\footnotesize $n-1$};}
                }
            }
        }

        \foreach \b [count=\a] in {4,...,7}{%
        \draw ([xshift=9.8cm] \a*360/4: 1.3cm) node [circle, text width=5mm, align=center] (\b) {};
        }
        \draw [draw=ForestGreen, ->, dashed] (4) edge node[pos=0.5, fill=white, inner sep=1mm] {\footnotesize $z_{2}$} (6); 
    
    \draw (10, -3) node {Case 3};
    \end{tikzpicture}
    \]
    where the shaded vertices are in $Q_{\hat{V}}$, the vertices marked with a $\times$ have been removed from $Q$ (i.e.\ they are not in $Q_{\hat{V}}$) and the vertex corresponding to a non-filled in circle means that it may or may not be in $Q_{\hat{V}}$.

    Hence, it suffices to show that the simple modules at each vertex in cases 1--3 have finite projective dimension. Throughout, let $S(i)$ denote the simple module at vertex $i$. Moreover, let $2s$ be the smallest even vertex in $V$ (i.e.\ it is the smallest even vertex not in $Q_{\hat{V}}$). If there is no even vertex in $V$ set $2s = \infty$. For each $i$, let $\delta_i = 0$  if $2s \leq i$ and $\delta_i = 1$ otherwise. 

    \begin{itemize}
        \item[\textit{Case 1.}] First, notice that $z_3$ is deleted from $Q$. So, given the commutativity relations in $I$, then $ \alpha_0 z_2, \ z_2 \alpha_1  \in I_{\hat{V}}$ for $\alpha_i \in \{x_i, y_i\}$. Keeping this in mind, we compute the minimal projective resolution of $S(0)$.

        The projective cover of $S(0)$ is $q_0 \colon P(0) \to S(0)$ and has kernel $\radd P(0)$. The module $P(0)$ can be computed as the quiver representation where
        {\small
        \[ P(0)_i = \begin{cases}
            \Span_k(\{e_0\}) & i = 0, \\
             \Span_k( \{\delta_i z_2 \ldots z_{i} \}) & 0 < i = 2j, \\
             0 & i = 2j + 1. \\
        \end{cases}
        \]}
        and so 
        {\small
        \[ \left( \radd P(0) \right)_i = \begin{cases}
             \Span_k( \{\delta_i z_2 \ldots z_{i} \}) & 0 < i = 2j, \\
             0 & \text{otherwise.}
        \end{cases}
        \]}
        It is straightforward to check that the projective cover of $\radd P(0)$ is $q_1 \colon P(2) \tow{z_2 \cdot} \radd P(0)$. Moreover, there is an isomorphism $P(1)^{\oplus 2} \underset{\sim}{\tow{(x_1 \cdot, y_1 \cdot)}} \ker q_1$. Hence, the minimal projective resolution of $S(0)$ is 
        \[ 0 \to P(1)^{\oplus 2} \tow{(x_1 \cdot, y_1 \cdot)} P(2) \tow{z_2 \cdot} P(0) \to S(0) \to 0. \]
        Following a similar strategy, it is straightforward to check that $\radd P(1) \cong S(0)^{\oplus 2}$ and so the minimal projective resolution of $S(1)$ is
        \[  0 \to P(1)^{\oplus 4} \to P(2)^{\oplus 2} \to P(0)^{\oplus 2} \to P(1) \to S(1) \to 0. \]
       
        Finally to compute the projective resolution of $S(2)$, we need to consider two cases: either $4 \in Q_{\hat{V}}$ or $4 \notin Q_{\hat{V}}$. In the latter case, it is easy to check that there is an isomorphism $P(1)^{\oplus 2} \xrightarrow[\sim]{(x_1 \cdot, y_1 \cdot)} \radd P(2)$ and so the minimal projective resolution of $S(2)$ is
        \[ 0 \to P(1)^{\oplus 2} \to P(2) \to S(2) \to 0. \]
        On the other hand, if $4 \in Q_{\hat{V}}$, then  $P(1)^{\oplus 2} \oplus P(4) \xrightarrow[\sim]{(x_1 \cdot, y_1 \cdot, z_4 \cdot)} \radd P(2)$ and so the projective resolution of $S(2)$ is
        \[ 0 \to P(1)^{\oplus 2} \oplus P(4) \to P(2) \to S(2) \to 0. \]
        Thus, all of the simple modules in Case 1 have finite projective dimension.
        \item[\textit{Case 2.}] In Case 2, notice that the vertex $0$ is a sink. Hence, $P(0) \cong S(0)$ and so it suffices to compute the minimal projective resolution of $S(1)$. To do so, we follow a similar strategy as in Case 1.
        
        If $3 \notin Q_{\hat{V}}$, then it is straightforward to check that $P(0)^{\oplus 2} \xrightarrow[\sim]{(x_0 \cdot, y_0 \cdot)} \radd P(1)$. Instead if $3 \in Q_{\hat{V}}$, then $P(0)^{\oplus 2} \oplus P(3) \xrightarrow[\sim]{(x_0 \cdot, y_0 \cdot, z_3 \cdot)} \radd P(1)$. In either case, $\radd P(1)$ is projective and so $S(1)$ has finite projective dimension. 
        \item[\textit{Case 3.}] For the third case, suppose first $2 \notin Q_{\hat{V}}$. Then, $P(0) \cong S(0)$ and we are done. Thus, suppose that $2 \in Q_{\hat{V}}$. Then, there is an isomorphism $P(2) \underset{\sim}{\tow{z_2 \cdot}} \radd P(0)$. In either case, $S(0)$ has finite projective dimension.  \qedhere
    \end{itemize}
 \end{proof}

It turns out that the converse of \ref{fin gdim} is also true. To prove this, the following lemma is key.

\begin{lemma} \label{min res Sn-1}
     Suppose that $Q_{\hat{V}}$ seen as a subquiver of $Q$ contains all vertices $0, \ n-1, \ \ldots, \ n-t$ for some $t>3$. Assume further that $1$ and $n-t-1$ are not contained in $Q_{\hat{V}}$. Then, the beginning of the minimal projective resolution of $S(n-1)$ is
     \begin{equation} \label{eq:min res Sn-1}
          0 \to (\radd P(n-1))^{\oplus 4} \oplus N \to P(0)^{\oplus 2} \tow{(z_0 \cdot, z_0 \cdot)} P(n-2)^{\oplus 2} \tow{(x_{n-2} \cdot, y_{n-2} \cdot)} P(n-1) \to S(n-1) \to 0.
     \end{equation}
\end{lemma}
\begin{proof}
    By assumption, the quiver $Q_{\hat{V}}$ can be drawn as 
    
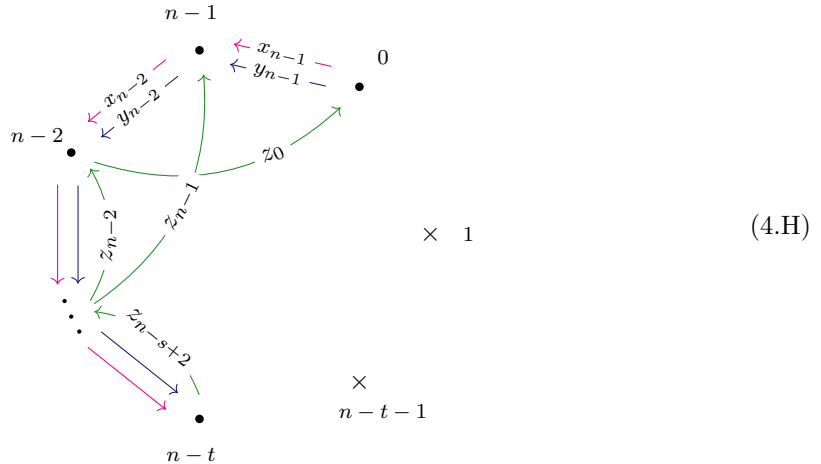
\begin{equation} \label{QV infinite gdim}
    \begin{tikzpicture}[baseline=(current  bounding  box.center)]
        \def\vradius{2.5cm}
        \def\lradius{3cm}
        \def\xradius{2.7cm}
        \def\yradius{2.4cm}
        \def\zradius{2.3cm}
        
        %Node 0
        \draw (1*360/7: \vradius) node [circle, fill=black, inner sep =0.4mm, align=center] (0) {};
        \draw (1*360/7: \lradius) node [circle, align=center]{\footnotesize 0};

        %Node n-1 and n-2 
        \foreach \a in {1,2}{%
            \draw (\inteval{\a+1}*360/7: \vradius) node [circle, fill=black, inner sep =0.4mm, align=center] (\a) {};
            \draw (\inteval{\a+1}*360/7: \lradius) node [circle, align=center]{\footnotesize $n-\a$};
        };

        %Node ...
        \draw (40*360/70: \vradius) node [circle, fill=black, inner sep =0.2mm, align=center] (3) {};
        \draw (39*360/70: \vradius) node [circle, fill=black, inner sep =0.2mm, align=center] (3a) {};
        \draw (41*360/70: \vradius) node [circle, fill=black, inner sep =0.2mm, align=center] (3b) {};
       % \draw (4*360/7: 2.2cm) node [circle, align=center]{\footnotesize $\vdots$};

       %Node n-k
        \draw (5*360/7: \vradius) node [circle, fill=black, inner sep =0.4mm, align=center] (4) {};
        \draw (5*360/7: \lradius) node [circle, align=center]{\footnotesize $n-t$};

        %Node n-k-1
        \draw (6*360/7: \vradius) node [circle, inner sep =0.4mm, align=center] (5) {$\times$};
        \draw (6*360/7: \lradius) node [circle, align=center]{\footnotesize $n-t-1$};

        %Node 1
        \draw (7*360/7: \vradius) node [circle, inner sep =0.4mm, align=center] (6) {$\times$};
        \draw (7*360/7: \lradius) node [circle, align=center]{\footnotesize 1};

        %Setting up nodes for arrow pos.
        \foreach \b [count=\a] in {7,...,13}{%
            \draw (\a*360/7: \xradius) node [circle, text width=7.5mm, align=center] (\b) {};
        }
        \foreach \b [count=\a] in {14,...,20}{%
            \draw (\a*360/7: \yradius) node [circle, text width=5mm, align=center] (\b) {};
        }
        \foreach \b [count=\a] in {21,...,27}{%
            \draw (\a*360/7: \zradius) node (\b) {};
        }

        %x arrows
        \foreach \i [count=\a] in {7,...,8}{%
            \draw [draw = magenta, ->] (\i) edge node[pos=0.5, fill=white, inner sep=1mm,sloped] {\footnotesize $x_{n-\a}$} (\inteval{\i+1}); 
        }
        \foreach \i in {9,...,10}{%
            \draw [draw = magenta, ->] (\i) edge node[pos=0.5, inner sep=1mm,sloped] {} (\inteval{\i+1}); 
        }

        %y arrows
        \foreach \i [count=\a] in {14,...,15}{%
            \draw [draw = MidnightBlue, ->] (\i) edge node[pos=0.5, fill=white, inner sep=1mm,sloped] {\footnotesize $y_{n-\a}$} (\inteval{\i+1}); 
        }
        \foreach \i in {16,...,17}{%
            \draw [draw = MidnightBlue, ->] (\i) edge node[pos=0.5, inner sep=1mm,sloped] { } (\inteval{\i+1}); 
        }

        %z arrows
        \draw [draw = ForestGreen, ->] (23) edge[bend right] node[pos=0.7, fill=white, inner sep=1mm,sloped] {$z_{0}$} (21); 
        \draw [draw = ForestGreen, ->] (24) edge[bend right] node[pos=0.5, fill=white, inner sep=1mm,sloped] {$z_{n-1}$} (22); 
        \draw [draw = ForestGreen, ->] (25) edge[bend right] node[pos=0.5, fill=white, inner sep=1mm,sloped] {$z_{n-s+2}$} (24); 
        \draw [draw = ForestGreen, ->] (24) edge[bend right] node[pos=0.5, fill=white, inner sep=1mm,sloped] {$z_{n-2}$} (23); 
    \end{tikzpicture}
    \end{equation}
    where we allow for the possibility that $n-t-1 = 1 \mod n$.
    
    Throughout, let $2s$ be the smallest even vertex in $V$ (i.e.\ it is the smallest even vertex not in $Q_{\hat{V}}$). If there is no even vertex in $V$ set $2s = \infty$. For each $i$, let $\delta_i = 0$  if $2s \leq i$ and $\delta_i = 1$ otherwise. 

    Notice first that the arrow $z_1$ is not in $Q_{\hat{V}}$, and so $\alpha_{n-2} z_0 \in I_{\hat{V}}$. As a result, all cycles at $0$, $n-1$ and $n-2$ vanish in $\C Q_{\hat{V}}/I_{\hat{V}}$.
    
    Keeping this in mind, it is straightforward to check that $P(0), P(n-1)$, and $P(n-2)$ are the quiver representations specified by
     \begin{fleqn}
    {\small
    \[ P(0)_i \quad \quad = \begin{cases}
        \Span_k(\{ e_{0} \}) & \text{$i = 0$}, \\
        \Span_k\left( \{ \delta _i z_2 \ldots z_i\} \cup \left\{ \alpha_3 \ldots \alpha_r \mid \alpha_j \in \{ x_{n-j}, y_{n-j} \} \right\} \right) & \text{$i=n-r$ even and $n-t \leq i$}, \\
        \Span_k\left( \left \{ \alpha_1 \ldots \alpha_r \mid \alpha_j \in \{ x_{n-j}, y_{n-j} \} \right\} \right) & \text{$i=n-r$ odd and $n-t \leq i$}, \\
        \Span_k(\{\delta_i z_2 \ldots z_i\}) & \text{$i$ even and $0< i < n-t$}, \\
        0 & \text{$i$ odd and $i < n-t$}.
    \end{cases}
    \]
    \[ P(n-1)_i = \begin{cases}
        \Span_k(\{ e_{n-1} \}) &  \text{$i = n-1$},\\
        \Span_k\left( \left \{ \alpha_2 \alpha_3 \ldots \alpha_r \mid \alpha_j \in \{ x_{n-j}, y_{n-j} \}, \right\} \right) \quad \quad \quad \quad \ \ & \text{$n-t \leq i = n - r < n - 1$}, \\
        0 & \text{$i < n-t$}.
    \end{cases}
    \]
    \[ P(n-2)_i = \begin{cases}
        \Span_k(\{ z_0 \})& \text{$i = 0$}, \\
        \Span_k(\{ z_0 x_{n-1}, z_0 y_{n-1} \} & \text{$i = n-1$},\\
        \Span_k(\{ e_{n-2} \}) & \text{$i = n-2$},\\
        \Span_k\left( \left\{ \delta_i z_0 \ldots z_i\} \cup \{ \alpha_3 \ldots \alpha_r \mid \alpha_j \in \{ x_{n-j}, y_{n-j} \} \right\} \right) & \text{$i=n-r$ even and $n-t \leq i < n-2$}, \\
        \Span_k(\{ \alpha_3 \ldots \alpha_r \mid \alpha_j \in \{ x_{n-j}, y_{n-j} \} ) & \text{$i=n-r$ odd and $n-t \leq i < n-1$}, \\
        \Span_k(\{ \delta_i z_0 \ldots z_i\}) & \text{$i$ even and $0 < i < n-t$}, \\ 
        0 & \text{$i$ odd and $i < n-t$}.
    \end{cases}
    \]
    }
    \end{fleqn}
    
    Thus,  $\radd P(n-1)$ is 
    {\small
    \[ \left( \radd P(n-1) \right)_i = \begin{cases}
        \Span_k\left( \left \{ \alpha_2 \alpha_3 \ldots \alpha_r \mid \alpha_j \in \{ x_{n-j}, y_{n-j} \} \right\} \right) & \text{$n-t \leq i = n - r < n - 1$}, \\
        0 & \text{otherwise}.
    \end{cases}
    \]}
    It now is straightforward to see that $q_1 \colon P(n-2)^{\oplus 2} \tow{(x_{n-2} \cdot, y_{n-2} \cdot)} \radd P(n-1)$ is the projective cover of $\radd P(n-1)$. The kernel of this map can be written as $\ker q_1 = A^{\oplus 2}$ where
    {\small
      \[ A_i = \begin{cases}
        \Span_k(\{ z_0 \})& \text{$i = 0$}, \\
        \Span_k(\{ z_0 x_{n-1}, z_0 y_{n-1} \} & \text{$i = n-1$},\\
        \Span_k\left( \{ \delta_i z_0 \ldots z_i\} \right) & \text{$i$ even and $i \neq n-1$}, \\
        0 & \text{otherwise}.
    \end{cases}
    \]}
    The projective cover of $A$ is $q_2 \colon P(0) \tow{z_0 \cdot} A$, the kernel of which is
        \[ (\ker q_2)_i \quad \quad = \begin{cases}
        \Span_k\left( \{ \delta_{n-1} z_2 \ldots z_{n-1}\} \right) & \text{$i = n-1$} \\
        \Span_k\left( \left \{ \alpha_1 \ldots \alpha_t \mid \alpha_j \in \{ x_{n-j}, y_{n-j} \} \right\} \right) & \text{$i=n-r$ and $n-t \leq i<n-1$}, \\
        0 & \text{otherwise}.
    \end{cases}
    \]
    Notice that $\ker q_2 = M \oplus N$ where $N_i = \Span_k\left( \{ \delta_{n-1} z_2 \ldots z_{n-1}\} \right)$ if $i = n-1$ and $N_i = 0$ otherwise and, further,
    \[ M_i \quad \quad = \begin{cases}
        \Span_k\left( \left \{ \alpha_1 \ldots \alpha_t \mid \alpha_i \in \{ x_{n-i}, y_{n-i} \} \right\} \right) & \text{$i=n-t$ and $n-s \leq i$}, \\
        0 & \text{otherwise}.
    \end{cases}
    \]
    The key point to notice is that there is an isomorphism $(\radd P(n-1))^{\oplus 2} \xrightarrow[\sim]{ (x_1 \cdot, y_1 \cdot)} M$. We have thus shown that start of the minimal projective resolution of $S(n-1)$ is \eqref{eq:min res Sn-1}.
\end{proof}

\begin{prop} \label{infinite gdim}
    Let $e \in R\#G$ be an idempotent not equal to $0$. If $Q_{\hat{V}}$ seen as a subquiver of $Q$ contains four successive vertices, then $R\#G/ \langle e \rangle$ has infinite global dimension.
\end{prop}
\begin{proof}
    Suppose that $Q_{\hat{V}}$ has a maximum of $t$ vertices in a row as a subquiver of $Q$. Then, there is a point in $Q_{\hat{V}}$ whose neighbourhood (up to a relabelling) looks like \eqref{QV infinite gdim}. Hence, by \ref{min res Sn-1} up to a relabelling,   the beginning of the minimal projective resolution of $S(n-1)$ is \eqref{eq:min res Sn-1}.
    
    Let $M$ be a module with minimal projective resolution $q \colon Q \to M$. Recall that the $i$th syzygy of $M$ is $\Omega^i M := \ker q^{(i-1)}$. A module has finite projective dimension if and only if $\Omega^i M = 0$ for some $i$. Thus, we will show that $S(n-1)$ has infinite projective dimension by showing that $\Omega^i S(n-1)$ never vanishes.

    Well, we can see from \eqref{eq:min res Sn-1} that  $(\radd P(n-1))^{\oplus 4}$ is a summand of $\Omega S(n-1)$. Proceeding inductively, we see that the dimension of $\Omega^i S(n-1)$ only increases with $i$ and therefore cannot vanish. Hence, necessarily $S(n-1)$ has infinite projective dimension.
\end{proof}

In summary, we can characterise precisely when $R\#G/ \langle e \rangle$ has finite global dimension.

\begin{theorem} \label{criterion}
    Let $e \in R\#G$ be an idempotent not equal to $0$. Then $R\#G/ \langle e \rangle$ has finite global dimension if and only if $Q_{\hat{V}}$ seen as a subquiver of $Q$ does not contain four successive vertices.
\end{theorem}

\begin{example}
    Consider the skew-group algebra of example \ref{mckay example 2}. Then, by \ref{criterion}, the idempotent quotients of $R\#G'$ which have finite global dimension are precisely the ones for which  $Q_{\hat{V}}$ seen as a subquiver of $Q$ does not contain four successive vertices. All possible such quotients (up to isomorphism) are presented in \cref{fig:quot idemp}. To avoid cluttered notation, we omit the arrow labels. For us, a dashed arrow into vertex $i$ represents a $z_i$ arrow. Two solid arrows into vertex $i$ represent arrows $x_i$ and $y_i$.
    
    By \ref{quot them} it follows that each algebra in \cref{fig:quot idemp} induces autoequivalences of $\Db(\modu R\#G')$ and hence of $\Db(\coh G' \text{-} \Hilb(\C^3))$. We may interpret this result as a statement that the categories $\Db(\modu R\#G')$ and $\Db(G' \text{-} \Hilb(\C^3)$ have large autoequivalence groups.

\newpage
        \begin{figure}[h]
        \centering
        
        \begin{tikzpicture}
            \def\vradius{1.1cm}
            \def\xradius{1.4cm}
            \def\yradius{1cm}
            \def\zradius{0.9cm}
            
       %%%%%%%%%%%%%%%%%%%%%%%%%%%%%%%%%%%% 6 + 2 %%%%%%%%%%%%%%%%%%%%%%%%%%%%%%%%%%%%
        \begin{scope}
        %vertices
        \draw (1*360/5: \vradius) node [circle, text width=5mm, align=center] (0) {0};
        \draw (2*360/5: \vradius) node [circle, text width=5mm, align=center] (1) {1};
        \draw (3*360/5: \vradius) node [circle, text width=5mm, align=center] (3) {3};
        \draw (4*360/5: \vradius) node [circle, text width=5mm, align=center] (4) {4};
        \draw (5*360/5: \vradius) node [circle, text width=5mm, align=center] (5) {5};

        \foreach \b [count=\a] in {10,...,14}{%
            \draw (\a*360/5: \xradius) node [circle, text width=9.5mm, align=center] (\b) {};
            \draw (\a*360/5: \yradius) node [circle, text width=5mm, align=center] (\inteval{\b+5}) {};
            \draw (\a*360/5: \zradius) node [circle, text width=5mm, align=center] (\inteval{\b+10}) {};
        }

        \draw [thick,draw = magenta, ->] (11) edge node[pos=0.5, inner sep = 1mm,sloped] {}  (10);
        \draw [thick,draw = magenta, ->] (13) edge node[pos=0.5,inner sep = 1mm,sloped] {}  (12);
        \draw [thick,draw = magenta, ->] (14) edge node[pos=0.5, inner sep = 1mm,sloped] {}  (13);

        \draw [thick, draw = MidnightBlue, ->] (16) edge node[pos=0.5, inner sep = 1mm,sloped] {}  (15);
        \draw [thick,draw = MidnightBlue, ->] (18) edge node[pos=0.5,inner sep = 1mm,sloped] {}  (17);
        \draw [thick,draw = MidnightBlue, ->] (19) edge node[pos=0.5, inner sep = 1mm,sloped] {}  (18);

        \draw [thick,draw = ForestGreen, ->, dashed] (5) edge node[pos=0.5, inner sep=1mm,sloped] {}(0); 
         \draw [thick,draw = ForestGreen, ->, dashed] (1) edge node[pos=0.5, inner sep=1mm,sloped] {}(3); 
         \draw [thick,draw = ForestGreen, ->, dashed] (22) edge node[pos=0.5, inner sep=1mm,sloped] {}(24); 

          \draw  (0, -2.2) node {\small (a) \ $R\#G' / \langle e_6 + e_2 \rangle$};
         \draw (-2.5, 1.2) node[anchor=west] {\small $Q_{ \hat{ \{6, 2\} } } \colon$ };
         \draw  (0, -1.6) node {\small $I_{ \hat{\{6, 2\}} }  = \langle x_3 z_5, y_3 z_5, z_5 x_4, z_5 y_4  \rangle$. };
         \end{scope}

         %%%%%%%%% 6 + 5 + 3 %%%%%%%%%%
         \begin{scope}[xshift=4.9cm]
        %vertices
        \draw (1*360/4: \vradius) node [circle, text width=5mm, align=center] (0) {0};
        \draw (2*360/4: \vradius) node [circle, text width=5mm, align=center] (1) {1};
        \draw (3*360/4: \vradius) node [circle, text width=5mm, align=center] (2) {2};
        \draw (4*360/4: \vradius) node [circle, text width=5mm, align=center] (4) {4};

        \foreach \b [count=\a] in {10,...,13}{%
            \draw (\a*360/4: \xradius) node [circle, text width=10.5mm, align=center] (\b) {};
            \draw (\a*360/4: \yradius) node [circle, text width=5mm, align=center] (\inteval{\b+4}) {};
            \draw (\a*360/4: \zradius) node [circle, text width=5mm, align=center] (\inteval{\b+8}) {};
        }

        \draw [thick, draw = magenta, ->] (11) edge node[pos=0.5, inner sep = 1mm,sloped] {}  (10);
        \draw [thick, draw = magenta, ->] (12) edge node[pos=0.5, inner sep = 1mm,sloped] {}  (11);

        \draw [thick, draw = MidnightBlue, ->] (15) edge node[pos=0.5,inner sep = 1mm,sloped] {}  (14);
        \draw [thick, draw = MidnightBlue, ->] (16) edge node[pos=0.5, inner sep = 1mm,sloped] {}  (15);

        \draw [thick, draw = ForestGreen, ->, dashed] (18) edge node[pos=0.5, inner sep=1mm,sloped] {}(20); 
         \draw [thick, draw = ForestGreen, ->, dashed] (2) edge node[pos=0.5, inner sep=1mm,sloped] {}(4);

         \draw  (0, -2.2) node {\small (b) \ $R\#G' / \langle e_6 + e_5 + e_3 \rangle$};
         \draw (-2.5, 1.2) node[anchor=west] {\small $Q_{ \hat{ \{6, 5, 3\} } } \colon$ };
         \draw  (0, -1.6) node {\small $I_{ \hat{\{6, 5, 3\}} }  = \langle x_0 z_2, y_0 z_2, z_2 x_1, z_2 y_1  \rangle$. };
         \end{scope}

         \begin{scope}[xshift=9.8cm]
        %vertices
        \draw (1*360/4: \vradius) node [circle, text width=5mm, align=center] (0) {0};
        \draw (2*360/4: \vradius) node [circle, text width=5mm, align=center] (1) {1};
        \draw (3*360/4: \vradius) node [circle, text width=5mm, align=center] (2) {2};
        \draw (4*360/4: \vradius) node [circle, text width=5mm, align=center] (5) {5};

        \foreach \b [count=\a] in {10,...,13}{%
            \draw (\a*360/4: \xradius) node [circle, text width=10.5mm, align=center] (\b) {};
            \draw (\a*360/4: \yradius) node [circle, text width=5mm, align=center] (\inteval{\b+4}) {};
            \draw (\a*360/4: \zradius) node [circle, text width=5mm, align=center] (\inteval{\b+8}) {};
        }

        \draw [thick, draw = magenta, ->] (11) edge node[pos=0.5, inner sep = 1mm,sloped] {}  (10);
        \draw [thick, draw = magenta, ->] (12) edge node[pos=0.5, inner sep = 1mm,sloped] {}  (11);

        \draw [thick, draw = MidnightBlue, ->] (15) edge node[pos=0.5,inner sep = 1mm,sloped] {}  (14);
        \draw [thick, draw = MidnightBlue, ->] (16) edge node[pos=0.5, inner sep = 1mm,sloped] {}  (15);

        \draw [thick, draw = ForestGreen, ->, dashed] (18) edge node[pos=0.5, inner sep=1mm,sloped] {}(20); 
         \draw [thick, draw = ForestGreen, ->, dashed] (5) edge node[pos=0.5, inner sep=1mm,sloped] {}(0); 

        \draw  (0, -2.2) node {\small (c) \ $R\#G' / \langle e_6 + e_4 + e_3 \rangle$};
         \draw (-2.5, 1.2) node[anchor=west] {\small $Q_{ \hat{ \{6, 4, 3\} } } \colon$ };
         \draw  (0, -1.6) node {\small $I_{ \hat{\{6, 4, 3\}} }  = \langle x_0 z_2, y_0 z_2, z_2 x_1, z_2 y_1  \rangle$. };
         \end{scope}

        %%%%%%%%%% 6 + 5 + 2 %%%%%%%%%
         \begin{scope}[yshift= -4.5cm]
        %vertices
        \draw (1*360/4: \vradius) node [circle, text width=5mm, align=center] (0) {0};
        \draw (2*360/4: \vradius) node [circle, text width=5mm, align=center] (1) {1};
        \draw (3*360/4: \vradius) node [circle, text width=5mm, align=center] (3) {3};
        \draw (4*360/4: \vradius) node [circle, text width=5mm, align=center] (4) {4};

        \foreach \b [count=\a] in {10,...,13}{%
            \draw (\a*360/4: \xradius) node [circle, text width=10.5mm, align=center] (\b) {};
            \draw (\a*360/4: \yradius) node [circle, text width=5mm, align=center] (\inteval{\b+4}) {};
            \draw (\a*360/4: \zradius) node [circle, text width=5mm, align=center] (\inteval{\b+8}) {};
        }

        \draw [thick, draw = magenta, ->] (11) edge node[pos=0.5, inner sep = 1mm,sloped] {}  (10);
        \draw [thick, draw = magenta, ->] (13) edge node[pos=0.5, inner sep = 1mm,sloped] {}  (12);

        \draw [thick, draw = MidnightBlue, ->] (15) edge node[pos=0.5,inner sep = 1mm,sloped] {}  (14);
         \draw [thick, draw = MidnightBlue, ->] (17) edge node[pos=0.5,inner sep = 1mm,sloped] {}  (16);

        \draw [thick, draw = ForestGreen, ->, dashed] (1) edge node[pos=0.5, inner sep=1mm,sloped] {}(3); 

        \draw  (0, -2.3) node {\small (d) \ $R\#G' / \langle e_6 + e_5 + e_2 \rangle$};
         \draw (-2.5, 1.2) node[anchor=west] {\small $Q_{ \hat{ \{6, 5, 2\} } } \colon$ };
         \draw  (0, -1.7) node {\small $I_{ \hat{\{6, 5, 2\}} }  = 0$. };
         \end{scope}

         \begin{scope}[xshift = 5cm, yshift= -4.5cm]
                     %vertices
        \draw (1*360/4: \vradius) node [circle, text width=5mm, align=center] (0) {0};
        \draw (2*360/4: \vradius) node [circle, text width=5mm, align=center] (1) {1};
        \draw (3*360/4: \vradius) node [circle, text width=5mm, align=center] (3) {3};
        \draw (4*360/4: \vradius) node [circle, text width=5mm, align=center] (5) {5};

        \foreach \b [count=\a] in {10,...,13}{%
            \draw (\a*360/4: \xradius) node [circle, text width=10.5mm, align=center] (\b) {};
            \draw (\a*360/4: \yradius) node [circle, text width=5mm, align=center] (\inteval{\b+4}) {};
            \draw (\a*360/4: \zradius) node [circle, text width=5mm, align=center] (\inteval{\b+8}) {};
        }

        \draw [thick, draw = magenta, ->] (11) edge node[pos=0.5, inner sep = 1mm,sloped] {}  (10);

        \draw [thick, draw = MidnightBlue, ->] (15) edge node[pos=0.5,inner sep = 1mm,sloped] {}  (14);

        \draw [thick, draw = ForestGreen, ->, dashed] (1) edge node[pos=0.5, inner sep=1mm,sloped] {}(3); 
        \draw [thick, draw = ForestGreen, ->, dashed] (3) edge node[pos=0.5, inner sep=1mm,sloped] {}(5); 
        \draw [thick, draw = ForestGreen, ->, dashed] (5) edge node[pos=0.5, inner sep=1mm,sloped] {}(0);

        \draw  (0, -2.3) node {\small (e) \ $R\#G' / \langle e_6 + e_4 + e_2 \rangle$};
         \draw (-2.5, 1.2) node[anchor=west] {\small $Q_{ \hat{ \{6, 4, 2\} } } \colon$ };
         \draw  (0, -1.7) node {\small $I_{ \hat{\{6, 4, 2\}} }  = 0$. };
         \end{scope}

            \def\vradius{0.8cm}
            \def\xradius{1.35cm}
            \def\yradius{0.75cm}
        %%%%%%%%%% 6 + 5 + 4 + 3 %%%%%%%%%%%%%%
        \begin{scope}[yshift=-4.5cm, xshift=9.8cm]
        \draw (1*360/3: \vradius) node [circle, text width=3mm, align=center] (0) {0};
        \draw (2*360/3: \vradius) node [circle, text width=3mm, align=center] (1) {1};
        \draw (3*360/3: \vradius) node [circle, text width=3mm, align=center] (2) {2};

        \foreach \b [count=\a] in {10,...,12}{%
            \draw (\a*360/3: \xradius) node [circle, text width=14mm, align=center] (\b) {};
            \draw (\a*360/3: \yradius) node [circle, text width=3.5mm, align=center] (\inteval{\b+3}) {};
        }

        \draw [thick, draw = magenta, ->] (11) edge node[pos=0.5, inner sep = 1mm,sloped] {}  (10);
        \draw [thick, draw = magenta, ->] (12) edge node[pos=0.5, inner sep = 1mm,sloped] {}  (11);

        \draw [thick, draw = MidnightBlue, ->] (14) edge node[pos=0.5,inner sep = 1mm,sloped] {}  (13);
         \draw [thick, draw = MidnightBlue, ->] (15) edge node[pos=0.5,inner sep = 1mm,sloped] {}  (14);

        \draw [thick, draw = ForestGreen, ->, dashed] (0) edge node[pos=0.5, inner sep=1mm,sloped] {}(2);

        \draw  (0, -2.3) node {\small (f) \ $R\#G' / \langle e_6 + e_5 + e_4 + e_3 \rangle$};
         \draw (-2.5, 1.2) node[anchor=west] {\small $Q_{ \hat{ \{6, 5, 4, 3\} } } \colon$ };
         \draw  (0, -1.7) node {\small $I_{ \hat{\{6, 5, 4, 3\}} }  = \langle x_0 z_2, y_0 z_2, z_2 x_1, z_2 y_1 \rangle$. };
            \end{scope}

         %%%%%%%%%% 6 + 5 + 4 + 2 %%%%%%%%%%%%%%
        \begin{scope}[yshift=-9cm]
        \draw (1*360/3: \vradius) node [circle, text width=3mm, align=center] (0) {0};
        \draw (2*360/3: \vradius) node [circle, text width=3mm, align=center] (1) {1};
        \draw (3*360/3: \vradius) node [circle, text width=3mm, align=center] (3) {3};

        \foreach \b [count=\a] in {10,...,12}{%
            \draw (\a*360/3: \xradius) node [circle, text width=14mm, align=center] (\b) {};
            \draw (\a*360/3: \yradius) node [circle, text width=3.5mm, align=center] (\inteval{\b+3}) {};
        }

        \draw [thick, draw = magenta, ->] (11) edge node[pos=0.5, inner sep = 1mm,sloped] {}  (10);

        \draw [thick, draw = MidnightBlue, ->] (14) edge node[pos=0.5,inner sep = 1mm,sloped] {}  (13);

        \draw [thick, draw = ForestGreen, ->, dashed] (1) edge node[pos=0.5, inner sep=1mm,sloped] {}(3);

        \draw  (0, -1.9) node {\small (g) \ $R\#G' / \langle e_6 + e_5 + e_4 + e_2 \rangle$};
         \draw (-2.5, 1.2) node[anchor=west] {\small $Q_{ \hat{ \{6, 5, 4, 2\} } } \colon$ };
         \draw  (0, -1.3) node {\small $I_{ \hat{\{6, 5, 4, 2\}} }  = 0$. };
            \end{scope}

        %%%%%%%%%% 6 + 5 + 4 + 3 %%%%%%%%%%%%%%
        \begin{scope}[yshift=-9cm, xshift=4.9cm]
        \draw (1*360/3: \vradius) node [circle, text width=3mm, align=center] (0) {0};
        \draw (2*360/3: \vradius) node [circle, text width=3mm, align=center] (2) {2};
        \draw (3*360/3: \vradius) node [circle, text width=3mm, align=center] (3) {3};

        \foreach \b [count=\a] in {10,...,12}{%
            \draw (\a*360/3: \xradius) node [circle, text width=14mm, align=center] (\b) {};
            \draw (\a*360/3: \yradius) node [circle, text width=3.5mm, align=center] (\inteval{\b+3}) {};
        }

        \draw [thick, draw = magenta, ->] (12) edge node[pos=0.5, inner sep = 1mm,sloped] {}  (11);

        \draw [thick, draw = MidnightBlue, ->] (15) edge node[pos=0.5,inner sep = 1mm,sloped] {}  (14);

        \draw [thick, draw = ForestGreen, ->, dashed] (0) edge node[pos=0.5, inner sep=1mm,sloped] {}(2);

        \draw  (0, -1.9) node {\small (h) \ $R\#G' / \langle e_6 + e_5 + e_4 + e_1 \rangle$};
         \draw (-2.5, 1.2) node[anchor=west] {\small $Q_{ \hat{ \{6, 5, 4, 1\} } } \colon$ };
         \draw  (0, -1.3) node {\small $I_{ \hat{\{6, 5, 4, 1\}} }  = 0 $. };
            \end{scope}

        %%%%%%%%%% 6 + 5 + 3 + 2 %%%%%%%%%%%%%%
        \begin{scope}[yshift=-9cm, xshift=9.8cm]
        \draw (1*360/3: \vradius) node [circle, text width=3mm, align=center] (0) {0};
        \draw (2*360/3: \vradius) node [circle, text width=3mm, align=center] (1) {1};
        \draw (3*360/3: \vradius) node [circle, text width=3mm, align=center] (4) {4};

        \foreach \b [count=\a] in {10,...,12}{%
            \draw (\a*360/3: \xradius) node [circle, text width=14mm, align=center] (\b) {};
            \draw (\a*360/3: \yradius) node [circle, text width=3.5mm, align=center] (\inteval{\b+3}) {};
        }

        \draw [thick, draw = magenta, ->] (11) edge node[pos=0.5, inner sep = 1mm,sloped] {}  (10);

        \draw [thick, draw = MidnightBlue, ->] (14) edge node[pos=0.5,inner sep = 1mm,sloped] {}  (13);

        \draw  (0, -1.9) node {\small (i) \ $R\#G' / \langle e_6 + e_5 + e_3 + e_2 \rangle$};
         \draw (-2.5, 1.2) node[anchor=west] {\small $Q_{ \hat{ \{6, 5, 3, 2\} } } \colon$ };
         \draw  (0, -1.3) node {\small $I_{ \hat{\{6, 5, 3, 2\}} }  = 0$. };
            \end{scope}

        %%%%%%%%%% 6 + 5 + 3 + 1 %%%%%%%%%%%%%%
        \begin{scope}[yshift=-13.1cm]
        \draw (1*360/3: \vradius) node [circle, text width=3mm, align=center] (0) {0};
        \draw (2*360/3: \vradius) node [circle, text width=3mm, align=center] (2) {2};
        \draw (3*360/3: \vradius) node [circle, text width=3mm, align=center] (4) {4};

        \draw [thick, draw = ForestGreen, ->, dashed] (0) edge node[pos=0.5, inner sep=1mm,sloped] {}(2);
        \draw [thick, draw = ForestGreen, ->, dashed] (2) edge node[pos=0.5, inner sep=1mm,sloped] {}(4);

        \draw  (0, -1.9) node {\small (j) \ $R\#G' / \langle e_6 + e_5 + e_3 + e_1 \rangle$};
         \draw (-2.5, 1.2) node[anchor=west] {\small $Q_{ \hat{ \{6, 5, 3, 1\} } } \colon$ };
         \draw  (0, -1.3) node {\small $I_{ \hat{\{6, 5, 3, 1\}} }  = 0$. };
            \end{scope}

                %%%%%%%%%%%% 6 + 5 + 4 + 3 + 2 + 1 %%%%%%%
        \begin{scope}[yshift=-13.1cm, xshift=4.9cm]
        \draw (1*360/3: \vradius) node [circle, text width=3mm, align=center] (0) {0};
        \draw (2*360/3: \vradius) node [circle, text width=3mm, align=center] (1) {1};
        
        \draw (1*360/3: \xradius) node [circle, text width=14mm, align=center] (2) {};
        \draw (2*360/3: \xradius) node [circle, text width=14mm, align=center] (3) {};

        \draw (1*360/3: \yradius) node [circle, text width=3.5mm, align=center] (4) {};
        \draw (2*360/3: \yradius) node [circle, text width=3.5mm, align=center] (5) {};

        \draw [thick, draw = magenta, ->] (3) edge node[pos=0.5, inner sep=1mm,sloped] {}(2);
        \draw [thick, draw = MidnightBlue, ->] (5) edge node[pos=0.5, inner sep=1mm,sloped] {}(4);

        \draw  (0, -1.9) node {\small (k) \ $R\#G' / \langle 1 - e_0 - e_1 \rangle$};
         \draw (-2.5, 1.2) node[anchor=west] {\small $Q_{ \hat{ \{6, 5, 4, 3, 1\} } } \colon$ };
         \draw  (0, -1.3) node {\small $I_{ \hat{\{6, 5, 4, 3, 1\}} }  = 0$. };
          \end{scope}

          \begin{scope}[xshift=9.8cm, yshift=-13.1cm]
        \draw (1*360/3: \vradius) node [circle, text width=3mm, align=center] (0) {0};
        \draw (2*360/3: \vradius) node [circle, text width=3mm, align=center] (2) {2};  

        \draw [thick, draw = ForestGreen, ->, dashed] (0) edge node[pos=0.5, inner sep=1mm,sloped] {}(2);
        
        \draw  (0, -1.9) node {\small (l) \ $R\#G' / \langle 1 - e_0 - e_2 \rangle$};
         \draw (-2.5, 1.2) node[anchor=west] {\small $Q_{ \hat{ \{6, 5, 4, 3, 1\} } } \colon$ };
         \draw  (0, -1.3) node {\small $I_{ \hat{\{6, 5, 4, 3, 1\}} }  = 0$. };
          \end{scope}

    \begin{scope}[yshift = - 16.6cm]
        \draw (0.5, 0) node [circle, text width=5mm, align=center] (0) {0};
        \draw (-1, 0) node [circle, text width=5mm, align=center] (3) {3};

        \draw (-2.5, 0.6) node[anchor=west] {\small $Q_{ \hat{ \{6, 5, 4, 2, 1\} } } \colon$ };
        \draw  (0, -1.2) node {\small (m) \ $R\#G' / \langle 1 - e_0 - e_3 \rangle$};;
         \draw  (0, -0.6) node {\small $I_{ \hat{\{6, 5, 4, 2, 1\}} }  = 0$. };
    \end{scope}
        \end{tikzpicture}
        
        \caption{All isomorphism classes of quotients $R\#G'/e$ which have finite global dimension. In each case, there is an isomorphism $R\#G'/e \cong \C Q_{\hat{V}}/I_{\hat{V}}$. }
        \label{fig:quot idemp}
    \end{figure}
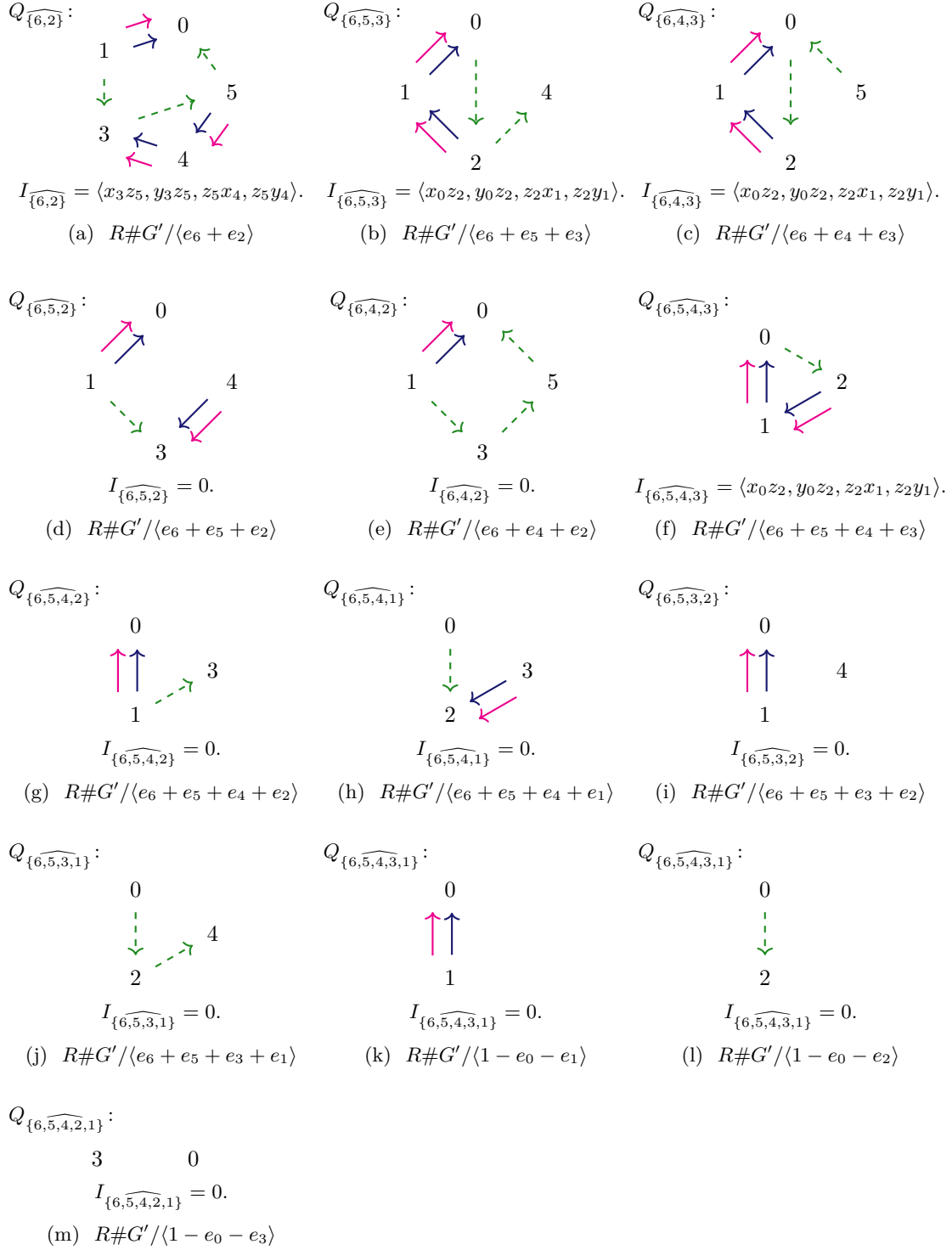
\end{example}
\appendix    
    
\printbibliography

\end{document}

%% file: 2_packages.tex
\providecommand{\lang}{UKenglish}
\providecommand{\geomoptions}{hmargin=3cm,vmargin=3.5cm}

\usepackage[utf8]{inputenc}

\usepackage[\lang]{babel}
\usepackage{csquotes}
\usepackage{lmodern}
\usepackage[T1]{fontenc}
\usepackage[\geomoptions]{geometry}
\usepackage{amsmath,amssymb,amsthm,amscd,mathrsfs,mathtools,nccmath,upgreek,euscript,dsfont,amsfonts,colonequals,graphicx, pifont,array}
\usepackage{adjustbox}

\usepackage[dvipsnames,svgnames]{xcolor}
\usepackage[pdflang=en-UK, colorlinks, urlcolor=Navy, linkcolor=Navy, citecolor=Navy]{hyperref}
\usepackage[nameinlink,noabbrev]{cleveref}
\crefname{section}{\S}{\S\S} 
\usepackage{tikz}
	\usetikzlibrary{calc}
	\usetikzlibrary{arrows,decorations.markings}
\usepackage{tikz-cd}
\tikzcdset{scale cd/.style={every label/.append style={scale=#1},
    cells={nodes={scale=#1}}}}
\usetikzlibrary{arrows}
\usetikzlibrary{shapes.geometric}

\usepackage{pgffor}
\usepackage{pstricks}
\usepackage{enumitem}
 
\input{4_macros.tex}

%% file: 4_macros.tex
\newcommand{\donothing}[1]{}

\setlist[enumerate]{format=\normalfont}

\renewcommand{\hat}{\widehat}
\renewcommand{\tilde}{\widetilde}

\renewcommand{\leq}{\leqslant}

\renewcommand{\alpha}{\upalpha}
\renewcommand{\beta}{\upbeta}
\renewcommand{\chi}{\upchi}
\renewcommand{\delta}{\updelta}
\renewcommand{\epsilon}{\upvarepsilon}
\renewcommand{\eta}{\upeta}
\renewcommand{\gamma}{\upgamma}
\renewcommand{\iota}{\upiota}
\renewcommand{\kappa}{\upkappa}
\renewcommand{\lambda}{\uplambda}
\renewcommand{\mu}{\upmu}
\renewcommand{\nu}{\upnu}
\renewcommand{\omega}{\upomega}
\renewcommand{\phi}{\upphi}
\renewcommand{\pi}{\uppi}
\renewcommand{\psi}{\uppsi}
\renewcommand{\rho}{\uprho}
\renewcommand{\sigma}{\upsigma}
\renewcommand{\tau}{\uptau}
\renewcommand{\theta}{\uptheta}
\renewcommand{\upsilon}{\upupsilon}

\renewcommand{\Delta}{\Updelta}
\renewcommand{\Gamma}{\Upgamma}
\renewcommand{\Lambda}{\Uplambda}
\renewcommand{\Omega}{\Upomega}
\renewcommand{\Psi}{\Uppsi}
\renewcommand{\Sigma}{\Upsigma}
\renewcommand{\Theta}{\Uptheta}
\renewcommand{\Upsilon}{\Upupsilon}
\renewcommand{\Xi}{\Upxi}

\DeclareMathOperator{\depth}{depth}

\DeclareMathOperator{\soc}{soc}
\DeclareMathOperator{\radd}{rad}

\newcommand{\idlm}{\mathfrak{m}}

\newcommand{\SL}{\mathrm{SL}}
\DeclareMathOperator{\Span}{Span}

\DeclareMathOperator{\Hom}{Hom}

\DeclareMathOperator{\End}{End}

\DeclareMathOperator{\uEnd}{\underline{End}}

\DeclareMathOperator{\Img}{Im}
\newcommand\Dtensor{\otimes^{\Lderived{}\kern-0.5em}
}

\newcommand{\op}{\mathrm{op}}
\newcommand{\bdd}{\mathrm{b}}

\DeclareMathOperator{\Dcat}{D}
\DeclareMathOperator{\Kcat}{K}
\newcommand\Db{\Dcat^\bdd}
\newcommand\Kb{\Kcat^\bdd}

\DeclareMathOperator{\modu}{mod}

\DeclareMathOperator{\proj}{proj}

\DeclareMathOperator{\add}{add}

\DeclareMathOperator{\fl}{f\kern -.5pt l}
\DeclareMathOperator{\CM}{CM}

\DeclareMathOperator{\Mod}{Mod}

\DeclareMathOperator{\coh}{coh}

\DeclareDocumentCommand \Rderived { o m } {%
  \IfNoValueTF {#1} {%
    {\rm\bf R}{#2}%
  }{%
    {\rm\bf R}^{#1}{#2}%
  }%
}

\DeclareDocumentCommand \Lderived { o m } {%
  \IfNoValueTF {#1} {%
    {\rm\bf L}{#2}%
  }{%
    {\rm\bf L}^{#1}{#2}%
  }%
}
\DeclareMathOperator{\Ext}{Ext}
\DeclareMathOperator{\Tor}{Tor}
\newcommand{\cohom}[1]{\mathrm{H}^{#1}}

\DeclareMathOperator{\cone}{cone}

\DeclareMathOperator{\pdim}{p.dim}
\DeclareMathOperator{\tot}{Tot}

\newcommand{\rmL}{\mathrm{L}}
\newcommand{\rmR}{\mathrm{R}}

\newcommand{\RA}{\mathrm{RA}}
\newcommand{\LA}{\mathrm{LA}}

\DeclareMathOperator{\spec}{Spec}
\DeclareMathOperator{\maxspec}{MaxSpec}

\newcommand{\con}{\mathrm{con}}
\newcommand*{\cHom}{\EuScript{H}\kern -.5pt \mathit{om}}
\newcommand*{\cEnd}{\EuScript{E}\kern -.5pt \mathit{nd}}

\DeclareMathOperator{\Hilb}{Hilb}

\newcommand{\tow}{\xrightarrow}

\newcommand{\N}{\mathbb{N}}
\newcommand{\Z}{\mathbb{Z}}

\newcommand{\C}{\mathbb{C}}

\newcommand{\cA}{\EuScript{A}}
\newcommand{\cB}{\EuScript{B}}
\newcommand{\cC}{\EuScript{C}}

\newcommand{\cN}{\EuScript{N}}

\newcommand{\cR}{\EuScript{R}}

\newcommand{\cT}{\EuScript{T}}

%% file: 1_biblatex.tex
\usepackage[backend=bibtex,citestyle=alphabetic,bibstyle=alphabetic,maxbibnames=99, giveninits=true,doi=false,isbn=false,url=false,eprint=true,useprefix=true]{biblatex}
\appto{\bibsetup}{\sloppy}

\DeclareFieldFormat{postnote}{#1}

\renewbibmacro{in:}{%
  \ifentrytype{article}{}{\printtext{\bibstring{in}\intitlepunct}}}

\DeclareFieldFormat[article,inbook,incollection,inproceedings,patent]{title}{#1}
\DeclareFieldFormat[thesis,unpublished]{title}{{\it #1}}

\DeclareFieldFormat[online]{date}{Preprint (#1)}

\renewbibmacro*{publisher+location+date}{%
  \printlist{publisher}%
  \setunit*{\addcomma\space}%
  \usebibmacro{date}%
  \newunit}
  
\DeclareFieldFormat{eprint:arxiv}{%
\ifentrytype{online}
  {\ifhyperref
    {\href{http://arxiv.org/abs/#1}{\nolinkurl{arXiv:#1}}}
    {\nolinkurl{arXiv:#1}}
   \iffieldundef{eprintclass}
    {}
    {{\tt \mkbibbrackets{\thefield{eprintclass}}}}}
  {\iffieldundef{eprintclass}
    {\mkbibparens{\ifhyperref
    {\href{http://arxiv.org/abs/#1}{\nolinkurl{arXiv:#1}}}
    {\nolinkurl{arXiv:#1}}}}
    {\mkbibparens{\ifhyperref
    {\href{http://arxiv.org/abs/#1}{\nolinkurl{arXiv:#1}}}
    {\nolinkurl{arXiv:#1}}
    {\tt \mkbibbrackets{\thefield{eprintclass}}}}}}}

%% file: 3_amsthm.tex
\theoremstyle{plain}
\newtheorem{theorem}{Theorem}[section]
\newtheorem{theorem*}{Theorem}
\newtheorem{lemma}[theorem]{Lemma}
\newtheorem{cor}[theorem]{Corollary}
\newtheorem{cor*}{Corollary}
\newtheorem{prop}[theorem]{Proposition}

\newtheorem{thmx}{Theorem}
\newtheorem{corx}[thmx]{Corollary}
\newtheorem{setx}[thmx]{Setup}
\theoremstyle{definition}
\newtheorem{definition}[theorem]{Definition}
\newtheorem{example}[theorem]{Example}

\newtheorem{setup}[theorem]{Setup}
\newtheorem{remark}[theorem]{Remark}
\newtheorem*{definition*}{Definition}

\theoremstyle{remark}
\newtheorem{notation}[theorem]{Notation}

\numberwithin{equation}{section}